\documentclass[11pt, reqno]{article}
\usepackage{amsmath,amssymb,amsfonts,amsthm}
\usepackage{color}

\usepackage[colorlinks=true,linkcolor=blue,citecolor=blue,urlcolor=blue,pdfborder={0 0 0}]{hyperref}
\usepackage{cleveref}

\usepackage{caption}
\usepackage{thmtools}

\usepackage{mathrsfs}

\crefname{theorem}{Theorem}{Theorems}
\crefname{thm}{Theorem}{Theorems}
\crefname{lemma}{Lemma}{Lemmas}
\crefname{lem}{Lemma}{Lemmas}
\crefname{remark}{Remark}{Remarks}
\crefname{prop}{Proposition}{Propositions}
\crefname{defn}{Definition}{Definitions}
\crefname{corollary}{Corollary}{Corollaries}
\crefname{conjecture}{Conjecture}{Conjectures}
\crefname{question}{Question}{Questions}
\crefname{chapter}{Chapter}{Chapters}
\crefname{section}{Section}{Sections}
\crefname{figure}{Figure}{Figures}

\theoremstyle{plain}
\newtheorem{thm}{Theorem}[section]

\newtheorem{lemma}[thm]{Lemma}

\newtheorem{corollary}[thm]{Corollary}

\newtheorem{prop}[thm]{Proposition}

\theoremstyle{definition}

\theoremstyle{remark}
\newtheorem{remark}[thm]{Remark}

\numberwithin{equation}{section}

\renewcommand{\P}{\mathbb P}
\newcommand{\E}{\mathbb E}

\newcommand{\R}{\mathbb R}
\newcommand{\Z}{\mathbb Z}
\newcommand{\N}{\mathbb N}

\newcommand{\cF}{\mathcal F}

\newcommand{\sA}{\mathscr A}
\newcommand{\sB}{\mathscr B}
\newcommand{\sC}{\mathscr C}
\newcommand{\sD}{\mathscr D}

\newcommand{\sF}{\mathscr F}
\newcommand{\sG}{\mathscr G}
\newcommand{\sH}{\mathscr H}

\newcommand{\sM}{\mathscr M}
\newcommand{\sN}{\mathscr N}

\usepackage[margin=1in]{geometry}
\usepackage{extarrows}
\usepackage{titling}
\usepackage{commath}
\usepackage{mathtools}
\usepackage{titlesec}
\newcommand{\eps}{\varepsilon}
\usepackage{bbm}
\usepackage{setspace}
\usepackage{enumitem}

\usepackage[textsize=tiny]{todonotes}

\usepackage{cite}

\newcommand{\bP}{\mathbf P}
\newcommand{\bE}{\mathbf E}

\newcommand{\myfrac}[3][0pt]{\genfrac{}{}{}{}{\raisebox{#1}{$#2$}}{\raisebox{-#1}{$#3$}}}

\newcommand{\Bridges}{\operatorname{Br}}

\def\P{\mathbb{P}}

\DeclareMathSymbol{\leqslant}{\mathalpha}{AMSa}{"36} 
\DeclareMathSymbol{\geqslant}{\mathalpha}{AMSa}{"3E} 
\DeclareMathSymbol{\eset}{\mathalpha}{AMSb}{"3F}     

\renewcommand{\epsilon}{\varepsilon}

\newcommand{\bD}{\mathbf{D}}
\newcommand{\bU}{\mathbf{U}}

\newcommand{\Rad}{\operatorname{Rad}}
\newcommand{\intrad}{\operatorname{Rad}_\mathrm{int}}

\pretitle{\begin{flushleft}\Large}
\posttitle{\par\end{flushleft}\vskip 0.5em
}
\preauthor{\begin{flushleft}}
\postauthor{
\\
\vspace{0.5em}
\footnotesize{The Division of Physics, Mathematics and Astronomy, California Institute of Technology.\\ Email: 
\href{mailto:t.hutchcroft@caltech.edu}{t.hutchcroft@caltech.edu}}
\end{flushleft}}
\predate{\begin{flushleft}}
\postdate{\par\end{flushleft}}
\title{\bf Slightly supercritical percolation on nonamenable graphs I:\\ The distribution of finite clusters
}

\renewenvironment{abstract}
 {\par\noindent\textbf{\abstractname.}\ \ignorespaces}
 {\par\medskip}

\titleformat{\section}
  {\normalfont\large \bf}{\thesection}{0.4em}{}

\author{{\bf Tom Hutchcroft}}
\begin{document}

\date{\small{\today}}

\maketitle

\setstretch{1.1}

\begin{abstract}
We study the distribution of finite clusters in slightly supercritical ($p \downarrow p_c$) Bernoulli bond percolation on transitive nonamenable graphs, proving in particular that if $G$ is a transitive nonamenable graph satisfying the \emph{$L^2$ boundedness condition} ($p_c<p_{2\to 2}$) and $K$ denotes the cluster of the origin then there exists $\delta>0$ such that if $p\in (p_c-\delta,p_c+\delta)$ then
\begin{align*}
\bP_p(n \leq |K| < \infty) &\asymp n^{-1/2} \exp\left[ -\Theta \!\left( |p-p_c|^2 n\right) \right]
\intertext{and}
\bP_p(r \leq \Rad(K) < \infty) &\asymp r^{-1} \exp\left[ -\Theta \Bigl( |p-p_c| r\Bigr) \right]
\end{align*}
for every $n,r\geq 1$, where all implicit constants depend only on $G$. We deduce in particular that the critical exponents $\gamma'$ and $\Delta'$ describing the rate of growth of the moments of a finite cluster as $p \downarrow p_c$ take their mean-field values of $1$ and $2$ respectively.

 These results apply in particular to Cayley graphs of nonelementary hyperbolic groups, to products with trees, and to transitive graphs of spectral radius $\rho<1/2$. In particular, every finitely generated nonamenable group has a Cayley graph to which these results apply. They are new for graphs that are not trees. 
The corresponding facts are yet to be understood on $\Z^d$ even for $d$ very large. In a second paper in this series, we will apply these results to study the geometric and spectral properties of infinite slightly supercritical clusters in the same setting.
\end{abstract}






\section{Introduction}\label{sec:intro}

In \textbf{Bernoulli bond percolation}, each edge of a countable graph $G=(V,E)$ is either deleted (\textbf{closed}) or retained (\textbf{open}) independently at random with retention probability $p\in[0,1]$ to obtain a random subgraph $\omega$ of $G$. The connected components of $\omega$ are referred to as \textbf{clusters}. We will be primarily interested  in the case that $G$ is \textbf{transitive}, i.e., that the automorphism group of $G$ acts transitively on $V$, or more generally that $G$ is \textbf{quasi-transitive}, i.e., that the action of the automorphism group of $G$ on $V$ has at most finitely many orbits. When $G$ is infinite, the \textbf{critical probability} $p_c=p_c(G)$ is defined by
\[
p_c=\sup\bigl\{p\in [0,1] : \text{ every cluster is finite $\bP_p$-a.s.}\bigr\},
\]
where we write $\bP_p=\bP_p^G$ for the law of Bernoulli-$p$ bond percolation on $G$. It is now known that the phase transition is \emph{non-trivial} (i.e., that $0<p_c<1$) for every infinite quasi-transitive graph with superlinear volume growth \cite{bperc96,1806.07733,lyons1995random}.


Percolation theorists are primarily interested in the geometry of clusters, and how this geometry changes as $p$ is varied. The theory is naturally decomposed into several regimes according to the relationship between $p$ and $p_c$. 
One possible taxonomy is as follows:
\begin{enumerate}
\item The \emph{subcritical} regime, in which $0< p < p_c$.
\item The \emph{slightly subcritical} regime, in which $0<p_c-p \ll 1$.
\item The \emph{critical} regime, in which $p=p_c$.
\item The \emph{slightly supercritical} regime, in which $0<p-p_c \ll 1$.
\item The \emph{supercritical} regime, in which $p_c<p < 1$.
\end{enumerate}
(It is sometimes desirable to further differentiate the \emph{very subcritical} regime $p \ll 1$ and \emph{very supercritical} regime $1-p \ll 1$; these regimes are often much easier to understand.)
Among all of these regimes, the most difficult to study is usually the \emph{slightly supercritical} regime.
A central difficulty in the study of this regime, and in supercritical percolation more generally, is that one is  interested in the probability of highly non-monotone events for the percolation configuration, such as $\{n \leq |K|<\infty\}$ where $K$ is the cluster of the origin, while many of the tools that have been developed in the study of the other regimes are either mostly or exclusively  suited to the analysis of monotone events and functions.


Indeed, there are essentially only two examples in which slightly supercritical percolation is reasonably well understood: trees\footnote{For $k$-regular trees and $p \geq p_c$, the conditional distribution of the cluster of the origin  given that it is finite is the same in Bernoulli-$p$ and Bernoulli-$q$ percolation, where $q$ is the unique to solution to $q (1-q)^{k-2}=p(1-p)^{k-2}$ lying in $[0,p_c]$. This duality is a consequence of the fact that every finite connected subgraph of a $k$-regular tree containing $n$ edges also touches exactly $(k-2)n+k$ edges that it does not contain.  As $p\downarrow p_c$, this dual probability $q$ satisfies  $p_c-q \sim p-p_c$. Thus, all questions concerning the distribution of finite clusters in slightly supercritical percolation can immediately be converted into questions concerning slightly subcritical percolation, which are much easier. This property is very specific to trees, and these arguments do not generalize to other nonamenable transitive graphs. Let us note, however, that slightly more involved duality arguments should also allow one to understand slightly supercritical percolation on transitive nonamenable \emph{proper plane} graphs with locally finite planar dual; to our knowledge such an analysis has not been carried out in the literature. Note that such graphs are always Gromov hyperbolic \cite{MR3658330} and therefore have $p_c<p_{2\to 2}$ by the results of \cite{1804.10191}. Thus, the results of this paper are always applicable to them.} and site percolation on the triangular lattice. In both cases, there are exact duality relations, developed extensively in the Euclidean setting by Kesten \cite{MR879034}, that allow us to convert questions about slightly supercritical percolation into questions about slightly \emph{subcritical} percolation. 
 In the case of trees these slightly subcritical questions can then be answered with the classical theory of branching processes (see e.g.\ \cite[Chapter 10]{grimmett2010percolation}), while for site percolation on the triangular lattice Smirnov and Werner \cite{MR1879816} showed that they can be answered by combining the aforementioned work of Kesten \cite{MR879034} with the theory of conformally invariant scaling limits and SLE as developed in the landmark works of Schramm \cite{Schramm00}, Smirnov \cite{MR1851632}, and Lawler, Schramm, and Werner \cite{MR1879850,MR1879851,MR1899232}. 
  This methodology is very specific to planar graphs, and does not give any indication of how these problems should be approached in higher-dimensional examples.


In particular, slightly supercritical percolation 
on $\Z^d$ remains poorly understood  even when $d$ is very large and all other regimes are now understood rather thoroughly. Highlights of the literature regarding the other regimes include 
\cite{MR852458,aizenman1987sharpness,duminil2015new} for the subcritical regime, \cite{MR1043524,MR762034,MR2748397,MR2551766} for the critical and slightly subcritical regimes, and \cite{MR1068308,MR1055419,MR1048927,MR594824} for the supercritical regime. See e.g.\ \cite{heydenreich2015progress,grimmett2010percolation,MR2241754} for overviews of this literature and of open problems in high dimensional percolation, and \cite{duminil2019upper} for some interesting recent partial progress on slightly supercritical percolation.  Let us also mention that a good understanding of slightly supercritical percolation appears to be a prerequisite to the solution of several important open problems regarding invasion percolation and minimal spanning forests, see \cite[Section 16.1]{heydenreich2015progress} and references therein.




The primary purpose of this series of two papers is to study slightly supercritical percolation in the `infinite-dimensional' setting of \emph{nonamenable} (quasi-)transitive  graphs. 
Here, we recall that a connected, locally finite graph is said to be \textbf{nonamenable} if its \textbf{Cheeger constant} 
\[
\Phi(G) = \inf \biggl\{\frac{|\partial_E W|}{\sum_{w\in W} \deg(w)}: W \text{ a finite set of vertices}\biggr\}
\]
is positive, where $\partial_E W$ denotes the set of edges with one endpoint in $W$ and one endpoint not in $W$; $G$ is said to be \textbf{amenable} if it is not nonamenable, i.e., if its Cheeger constant is zero. Background on percolation in the nonamenable context may be found in e.g.\ \cite{LP:book}. 
%
%
%
We prove our results under the additional hypothesis  that $G$ satisfies the \emph{$L^2$ boundedness condition}, which was introduced in \cite{1804.10191} and studied further in \cite{hutchcroft20192}. Let us now briefly introduce this condition. Given a countable graph $G=(V,E)$, we write $T_p(u,v)=\bP_p(u \leftrightarrow v)$ for the \textbf{two-point matrix}, and define 
\[
p_{2\to 2}=p_{2\to2}(G)=\sup\Bigl\{p\in [0,1]: \|T_p\|_{2\to 2} < \infty\Bigr\},
\]
where we recall that if $M \in [0,\infty]^{V^2}$ is a $V$-indexed matrix with non-negative entries then the $L^2(V) \to L^2(V)$ operator norm $\|M\|_{2\to2} \in [0,\infty]$ is defined by
\[
\|M\|_{2\to 2} = \sup\Bigl\{\|Mf\|_2 : f \in L^2(V),\, \|f\|_2 =1\Bigr\}.
\]
We say that $G$ satisfies the \emph{$L^2$ boundedness condition} if $p_c(G)<p_{2\to2}(G)$. This condition is conjectured to hold for every connected, locally finite, nonamenable quasi-transitive graph \cite[Conjecture 1.3]{hutchcroft20192}, and is now known to hold for several classes of examples, including Gromov hyperbolic graphs \cite{1804.10191}, highly nonamenable graphs \cite{MR1756965,MR1833805,MR3005730}, and graphs admitting a quasi-transitive nonunimodular  subgroup of automorphisms \cite{Hutchcroftnonunimodularperc}. In particular, it can be deduced by the methods of \cite{MR1756965} that every nonamenable, finitely generated group has a Cayley graph for which $p_c<p_{2\to 2}$. (On the other hand, we always have that $p_c=p_{2\to 2}$ in the amenable case.) See \cite{hutchcroft20192} for an overview. See also \cite{1808.08940,BLPS99b} and references therein for an overview of what is known regarding critical and near-critical percolation on general nonamenable transitive graphs without this assumption.


The main results of this paper apply the $L^2$ boundedness condition to establish a very precise understanding of the distribution of \emph{finite} clusters in critical and near critical percolation. In a forthcoming pair of sequel to this paper \cite{slightlysupercritical2,slightlysupercritical3}, we apply these results to study the large-scale geometry of \emph{infinite} clusters in slightly supercritical percolation. 
All of our results regarding slightly supercritical percolation are new when the graph in question is not a tree. 


The results of both papers build upon the methods of our recent work with Hermon \cite{HermonHutchcroftSupercritical}, which established related, non-quantitative results for supercritical percolation on nonamenable transitive graphs (that do not necessarily satisfy the $L^2$ boundedness condition). Making these arguments quantitative in a sharp way in order to get the correct behaviour as $p \downarrow p_c$ is a surprisingly delicate matter, and our proofs are, unfortunately, substantially more technical than those of \cite{HermonHutchcroftSupercritical}. 


Besides the intrinsic interest of our results, we are also hopeful that some of the tools we develop will be useful for approaching the high-dimensional Euclidean case; some perspectives on the remaining challenges in this case are presented in \cref{subsec:Euclidean}. It would also be very interesting (and seemingly highly non-trivial) to extend our methods to other infinite-dimensional settings, such as hypercubes or expander graphs (which are finite analogues of nonamenable graphs). Critical and slightly subcritical percolation on these graphs has been studied in many works, surveyed in \cite{MR3152019}, the highlights of which include \cite{MR2155704,MR2165583,MR2260845,MR3612867,hulshof2019slightly}. (The analogous results for the \emph{complete graph} are classical, see \cite{MR1864966} and references therein.)

\subsection{Statement of results}
\label{subsec:intro_finite}


We now state our results concerning the distribution of finite clusters in near critical percolation. While the supercritical aspects of these results are the most novel, it seems that they also improve slightly upon the best existing estimates for slightly subcritical percolation. We write $K_v$ for the cluster of $v$ and $|K_v|$ for the number of vertices it contains.


\begin{thm}[Volume of finite clusters] Let $G=(V,E)$ be a connected, locally finite, quasi-transitive graph such that $p_c(G)<p_{2\to 2}(G)$. Then there exists a constant $\delta=\delta(G)>0$ such that
\begin{align}
\bP_p\bigl(n \leq |K_v| < \infty\bigr) &\asymp n^{-1/2} \exp\left[ -\Theta \!\left( |p-p_c|^2 n\right) \right]
\label{eq:thmvolume}
\end{align}
for every $n \geq 1$, $v\in V$, and $p\in (p_c-\delta,p_c+\delta)$, where all implicit constants depend only on $G$.
\label{thm:main_volume}
\end{thm}

Here and below, we write $\asymp$, $\succeq$, and $\preceq$  to denote equalities and inequalities that hold up to positive multiplicative constants depending only on the graph $G$. Thus, for example, ``$f(n) \asymp g(n)$ for every $n\geq 1$'' means that there exist positive constants $c$ and $C$ such that $cg(n) \leq f(n) \leq Cg(n)$  for every $n\geq 1$. We use Landau's asymptotic notation similarly, so that, for example, $f(n)=\Theta(g(n))$  if and only if $f \asymp g$, and $f(n) \preceq g(n)$ if and only if $f(n) = O(g(n))$. In particular, \cref{thm:main_volume} is equivalent to the assertion that there exist positive constants $c_1$, $c_2$, $C_1$, $C_2$, and $\delta$ such that
\[
c_1 n^{-1/2} \exp\left[ -C_1 |p-p_c|^2 n\right] \leq \bP_p\bigl(n \leq |K_v| < \infty\bigr) \leq C_2 n^{-1/2} \exp\left[ -c_2 |p-p_c|^2 n\right]
\]
for every $v\in V$, $p\in (p_c-\delta,p_c+\delta)$, and $n\geq 1$.

\medskip

Our next theorem establishes a similar result for the \emph{radius} of a finite supercritical cluster.
We write $\Rad_\mathrm{int}(K_v)$ and $\Rad_\mathrm{ext}(K_v)$ for the intrinsic and extrinsic radii of $K_v$, that is, the maximum distance from $v$ to another point of $K_v$ in the graph metric on $K_v$ and in the graph metric on $G$ respectively. Note that we trivially have $\Rad_\mathrm{ext}(K_v)\leq \Rad_\mathrm{int}(K_v)$.

\begin{thm}[Radii of finite clusters] Let $G=(V,E)$ be a connected, locally finite, quasi-transitive graph such that $p_c(G)<p_{2\to 2}(G)$. Then there exists a constant $\delta=\delta(G)>0$ such that
\begin{align}
\bP_p\bigl(r \leq \Rad_\mathrm{int}(K_v) < \infty\bigr) &\asymp r^{-1} \exp\left[ -\Theta \Bigl( |p-p_c| r\Bigr) \right]
\label{eq:thmintrinsic}
 \intertext{and}
\bP_p\bigl(r \leq \Rad_\mathrm{ext}(K_v) < \infty\bigr) &\asymp r^{-1} \exp\left[ -\Theta \Bigl( |p-p_c| r\Bigr) \right]
\label{eq:thmextrinsic}
\end{align}
for every $r \geq 1$, $v\in V$, and $p\in (p_c-\delta,p_c+\delta)$, where all implicit constants depend only on $G$.
\label{thm:main_radius}
\end{thm}

The parts of these results concerning \emph{critical} percolation were already known, and are applied as a component of the proof. Indeed, a connected, locally finite, quasi-transitive graph $G$ is said to satisfying the \emph{triangle condition} if
\[\nabla_{p_c}(v) := \sum_{u,w \in V} T_{p_c}(v,u)T_{p_c}(u,w)T_{p_c}(w,v) < \infty
\]
for every $v\in V$. The triangle condition was introduced by Aizenman and Newman \cite{MR762034} and proven to hold on $\Z^d$ with $d$ large in the groundbreaking work of Hara and Slade \cite{MR1043524}. It is conjectured to hold if and only if $d>6$, and is now known to hold for all $d \geq 11$ \cite{fitzner2015nearest}. It is known that if a connected, locally finite, quasi-transitive graph $G$ satisfies the triangle condition then 
\begin{align}
\bP_{p_c}( |K_v| \geq n) &\asymp n^{-1/2} &&\qquad \text{ for every $n\geq 1$ and $v\in V$, and that}
\label{eq:volume_critical}
 \\
\bP_{p_c}(\intrad(K_v) \geq r ) &\preceq r^{-1} &&\qquad \text{ for every $r\geq 1$ and $v\in V$,}
\label{eq:int_critical}
\intertext{
so that, in particular, every cluster is finite $\bP_{p_c}$-almost surely. Note that the triangle condition is equivalent to the assertion that $T^3_{p_c}(v,v)< \infty$  for every $v \in V$ (powers of non-negative infinite matrices indexed by $V$ always being well-defined as elements of $[0,\infty]^{V\times V}$), and is therefore implied by the $L^2$  boundedness condition since $T^3_p(v,v) \leq \|T_p^3\|_{2\to 2} \leq \|T_p\|^3_{2\to 2}$.
 The upper and lower bounds of \eqref{eq:volume_critical} follow from the work of Aizenman and Newman \cite{MR762034} and Aizenman and Barsky \cite{aizenman1987sharpness} respectively, while \eqref{eq:int_critical} follows from the work of Kozma and Nachmias \cite{MR2551766}. A simple proof of the complementary lower bound $\bP_{p_c}( \intrad(K_v)\geq r) \succeq r^{-1}$, which holds on every connected, locally finite, quasi-transitive graph, is given in \cref{lem:subcritical_radius_lower}. Moreover, in \cite{hutchcroft20192} it is shown that the $L^2$ boundedness condition allows one to compare intrinsic and extrinsic distances, which allows one to prove in particular that
} 
\bP_{p_c}( \Rad_\mathrm{ext}(K_v) \geq r ) &\asymp r^{-1} &&\qquad \text{ for every $r\geq 1$ and $v\in V$.}
\label{eq:ext_critical}
\end{align}
for every connected, locally finite, quasi-transitive graph $G$ satisfying the $L^2$ boundedness condition. (On the other hand, Kozma and Nachmias \cite{MR2748397} proved that $\bP_{p_c}( \Rad_\mathrm{ext}(K_v) \geq r ) \asymp r^{-2}$ for percolation on $\Z^d$ with $d$ large. The disparity between these two results is related to the fact that random walk is diffusive on $\Z^d$ and ballistic on nonamenable graphs.)

Let $E(K_v)$ be the set of edges that \emph{touch} (i.e., have at least one endpoint in) $K_v$, and define
\begin{equation}
\label{eq:zeta_def}
\zeta(p) = -\limsup_{n\to \infty}  \frac{1}{n} \log \bP_p\bigl(n \leq |E(K_v)| < \infty\bigr)
\end{equation}
to be the exponential rate of decay of the probability that $v$ belongs to a large finite cluster (which is easily seen not to depend on the choice of $v$).
 It is a consequence of the \emph{sharpness of the phase transition} that $\zeta(p)>0$ for \emph{every} connected, locally finite, quasi-transitive graph $G$ and every $0\leq p<p_c$. This was first proven by Aizenman and Barsky \cite{aizenman1987sharpness} and Aizenman and Newman \cite{MR762034} (see also the closely related work of Menshikov \cite{MR852458}), and several alternative proofs are now available \cite{duminil2015new,MR3898174,1901.10363}. On the other hand, for \emph{supercritical} percolation on quasi-transitive graphs, it is shown in \cite{HermonHutchcroftSupercritical} that $p_c<1$ and $\zeta(p)>0$ for \emph{some} $p_c < p < 1$ if and only if $p_c<1$ and $\zeta(p)>0$ for \emph{every} $p_c < p < 1$, if and only if $G$ is nonamenable. (In contrast, finite supercritical clusters in $\Z^d$ have stretched-exponential volume tails \cite{MR1068308,MR1055419}.) Thus, for connected, locally finite, nonamenable quasi-transitive graphs, we have that $\zeta(p)>0$ if and only if $p \neq p_c$. (See \cite{MR4303886} for counterexamples regarding extensions of this result to dependent percolation models.) Note however that these arguments do not give any quantitative control on the manner in which $\zeta(p) \to 0$ as $p \to p_c$. \cref{thm:main_volume} provides such a quantitative understanding, and yields in particular the following immediate corollary.

\begin{corollary}
 Let $G=(V,E)$ be a connected, locally finite, quasi-transitive graph such that $p_c(G)<p_{2\to 2}(G)$. Then there exists $\delta>0$ such that $\zeta(p) \asymp |p-p_c|^2$ for every $p\in (p_c-\delta,p_c+\delta)$.
\end{corollary}

 Of course, \cref{thm:main_volume,thm:main_radius} tell us rather more than this:
they show us the precise manner in which the polynomial tail at $p_c$ is gradually transformed into the exponential tail away from $p_c$. In particular, 
 they make  the following natural heuristic picture precise: There is a \emph{scaling window} of order $|p-p_c|^{-1}$ such that within the scaling window percolation behaves in essentially the same way as critical percolation, whereas outside the scaling window the off-critical effects begin to become apparent.
Moreover, roughly speaking, these off-critical effects manifest themselves in a way that is proportional to how much larger our cluster is than a cluster that is at the edge of the scaling window (i.e., than a cluster that has radius $|p-p_c|^{-1}$ or volume $|p-p_c|^{-2}$). This intuitive picture will be an important motivation to many of our proofs: We will often prove estimates by separate analyses of the `inside-window' and `outside-window' cases. Note that the restriction to a neighbourhood of $p_c$ is necessary as $\zeta(p) \to \infty$ as $p\downarrow 0$ or $p \uparrow 1$.

\medskip

Finally, we note that \cref{thm:main_volume} also permits immediate computation of the slightly supercritical scaling exponents $\gamma'$ and $\Delta'$. It is believed that for every connected, locally finite, quasi-transitive graph $G=(V,E)$ there exist $\gamma,\gamma',\Delta,$ and $\Delta'$ such that if $k$ is a positive integer then
\begin{align}
\bE_p\left[ |K_v|^k \right] &\asymp_k |p-p_c|^{-\gamma-(k-1)\Delta \pm o_k(1)} &&\text{as $p \uparrow p_c$ and}\\
\bE_p\left[ |K_v|^k \mathbbm{1}(|K_v|<\infty)\right] &\asymp_k |p-p_c|^{-\gamma'-(k-1)\Delta' \pm o_k(1)} &&\text{as $p \downarrow p_c$,}
\end{align}
where the $k$ subscripts mean that the implicit constants may depend on $k$.
See \cite[Chapters 9 and 10]{grimmett2010percolation} for background on this conjecture. It is known that if $G$ satisfies the triangle condition then $\gamma$ and $\Delta$ are well-defined and take their \emph{mean-field} values of $1$ and $2$ respectively \cite{MR762034,MR923855} (see also \cite{1901.10363}). \cref{thm:main_volume} implies a similar result for $\gamma'$ and $\Delta'$ for graphs satisfying the $L^2$ boundedness condition.

\begin{corollary}
\label{cor:moments}
Let $G$ be a connected, locally finite, quasi-transitive graph such that $p_c(G)<p_{2\to 2}(G)$. Then there exist positive constants $\delta=\delta(G),c=c(G),$ and $C=C(G)$ such that
\[
c^k|p-p_c|^{-2k+1}k! \leq \bE_p\left[ |K_v|^k \mathbbm{1}(|K_v|<\infty) \right] \leq C^k|p-p_c|^{-2k+1}k!
\]
for every $k\geq 1$, $p\in (p_c-\delta,p_c)\cup(p_c,p_c+\delta)$, and $v\in V$. In particular, the exponents $\gamma=\gamma'=1$ and $\Delta=\Delta'=2$ are well-defined and take their mean-field values.
\end{corollary}

\subsection{About the proofs and organization}

Let us now outline the content of the rest of the paper, and in particular how the strategy we pursue here builds upon that of \cite{HermonHutchcroftSupercritical}.

\begin{enumerate}[leftmargin=*]
\item In \cref{sec:inside_window}, we prove some estimates on percolation `inside the scaling window', which in particular establish the upper bounds of \cref{thm:main_volume,thm:main_radius} in the cases $n = O(|p-p_c|^{-2})$ and $r = O(|p-p_c|^{-1})$ respectively. These estimates are straightforward applications of what is known about critical percolation under the triangle condition, and will be very useful in the remainder of our analysis.
\item In \cref{sec:outside_window}, we complete the proofs of the upper bounds of \cref{thm:main_volume,thm:main_radius}. This section takes up most of the paper, and is both the most technical and the most novel part of the paper. We pursue a similar strategy to that of \cite{HermonHutchcroftSupercritical}, but apply the assumption $p_c<p_{2\to 2}$ to obtain sharp quantitative versions of every estimate along the way.
\begin{enumerate}
\item In \cref{subsec:setup}, we recall some basic ideas and notation from \cite{HermonHutchcroftSupercritical} which allow us to express the derivative of, say, the truncated $k$th moment $\bE_{p,n}|K|^k := \bE_p |K|^k \mathbbm{1}(|K|\leq n)$ of the cluster volume as the difference of two terms: a `positive term' $\bU_{p,n}[|K|^k]$ which accounts for the effect of a finite cluster growing but remaining smaller than the truncation threshold $n$ and a `negative term' $-\bD_{p,n}[|K|^k]$ which accounts for the effect of finite clusters growing to break the truncation threshold $n$ (possibly by becoming infinite).

 Very roughly speaking, our goal in the remainder of the section will be i) to lower bound the absolute value of the negative term; ii) to write down an inequality of the form
\begin{align}
\bU_{p,n}\left[|K|^k\right] \leq \frac{1}{2}\bD_{p,n}\left[|K|^k\right] + \left(
\begin{array}l \text{something we can hope to bound without} \\ \text{yet understanding the truncated moments} \end{array}\right)
\nonumber
\end{align}
so that
\begin{multline}
\label{eq:overview_hope}
\frac{d}{dp}\bE_{p,n}\left[|K|^k\right] \\\leq -\frac{1}{2}\bD_{p,n}\left[|K|^k\right] + \left(
\begin{array}l \text{something we can hope to bound without} \\ \text{yet understanding the truncated moments} \end{array}\right);
\end{multline}
iii) to prove an upper bound on the second term on the right hand side of \eqref{eq:overview_hope} that is good enough to push through the final stage of the analysis. In the above formulation this would mean an upper bound of the form $k! C^k (p-p_c)^{-2k+1}$ for $p$ slightly larger than $p_c$; iv) to analyse the resulting differential inequality \eqref{eq:overview_hope} for $\bE_{p,n}|K|^k$.

\item In \cref{subsec:negative_term}, we carry out step (i) of this strategy, applying the $L^2$ boundedness condition to prove a lower bound on the magnitude of the negative term when $p$ is slightly larger than $p_c$, proving in particular that $\bD_{p,n}[|K|^k] \succeq (p-p_c) \bE_{p,n}[|K|^{k+1}]$ for $p$ slightly larger than $p_c$. This strengthens \cite[Proposition 2.4]{HermonHutchcroftSupercritical}, which established a similar but non-quantitative inequality for all transitive nonamenable graphs; the method of proof here is quite different.

\item In \cref{subsec:positive_term}, we carry out the remainder of the strategy but for the \emph{radius} rather than the volume, which is much easier. In this case, the analogue of the second term on the right hand side of \eqref{eq:overview_hope} is expressed in terms of the probability that the origin is in a large \emph{skinny} cluster, whose radius is large but whose volume is smaller than it ought to be given this large radius. An important part of the analysis is to obtain a sharp quantitative upper bound on  the probability of this event, which we will also apply many more times throughout the paper. This inequality can be thought of as a strengthening of \cite[Lemma 2.8]{HermonHutchcroftSupercritical} under the assumption that $\nabla_{p_c}<\infty$.

\item In \cref{subsec:positivetermI} and \cref{subsec:positivetermII}, we carry out the remainder of the strategy in the more difficult case of the volume. Here, the second term in \eqref{eq:overview_hope} is expressed in terms of clusters that satisfy a certain `higher-order' variation of the skinniness constraint considered above, related to the size of the tree of geodesics connecting $k+1$ points. In order to bound the resulting quantities, we introduce in \cref{subsec:positivetermI} a sequence of multivariate generating functions and prove that these generating functions satisfy a family of recursive differential inequalities relating the partial derivatives of $k$th function in the sequence to the value of the first $k$ functions in the sequence. In the following subsection \cref{subsec:positivetermII} we analyse this family of differential inequalities and then apply the resulting bounds to conclude the proof of the upper bounds of \cref{thm:main_volume} in the slightly supercritical case, i.e., to carry out step (iv) above. 

While these sections are based on similar high-level ideas to \cite[Section 2.3]{HermonHutchcroftSupercritical}, a much more delicate and technical implementation of these ideas was required to obtain sharp quantitative estimates. Indeed, while the methods developed in \cite[Section 2.3]{HermonHutchcroftSupercritical} are quantitative, they are not sharp, and eventually lead to estimates of the form $\zeta(p) \succeq (p-p_c)^4$ rather than $\zeta(p) \succeq (p-p_c)^2$ when fed the estimates of \cref{subsec:negative_term,subsec:positive_term} as inputs. In particular, while the family of differential inequalities between generating functions we derive here is closely related to \cite[Lemma 2.9]{HermonHutchcroftSupercritical}, the analysis it requires is completely different.
\end{enumerate}
\item In \cref{sec:completing} we complete the proofs of \cref{thm:main_volume,thm:main_radius} by proving lower bounds in the slightly supercritical regime as well as both upper and lower bounds in the critical and slightly subcritical regimes. While several of these estimates are fairly similar to things that are already known, a careful treatment is required to establish optimal quantitative forms of all the required estimates, and some of the results we prove here improve upon what was already known about slightly subcritical percolation under the triangle condition. Several of these sharp quantitative bounds are obtained with the help of the bounds on skinny clusters that are proven in \cref{subsec:positive_term}.
\item In \cref{subsec:Euclidean} we give some concluding remarks, including a  discussion of the challenges that remain to adapt our methods to the high-dimensional Euclidean case and some potential approaches to tackle them. 
\end{enumerate}

\noindent A glossary of notation is given at the end of the paper.

\begin{remark}
If the reader is familiar with \cite{HermonHutchcroftSupercritical}, they may notice that we do \emph{not} use one of the key ideas of that paper. In that paper, we wrote down a second formula for the derivative of the truncated $k$th moment in terms of the fluctuation of the number of open and closed edges in the cluster, the absolute value of which can be bounded via martingale methods. By comparing these bounds to those derived via Russo's formula as above, we were able to bound the truncated moments directly without actually analyzing the resulting differential inequalities. The reason we do not use this method here is that the bounds they yield are not sharp, but rather contain various unwanted polylogarithmic errors. In order to circumvent this issue we must actually analyze the differential inequalities we establish for the truncated moments.
\end{remark}

\section{Upper bounds inside the scaling window}
\label{sec:inside_window}

The purpose of this section is to prove the following lemma, which establishes the upper bounds of \cref{thm:main_volume,thm:main_radius} in the case that $n$ and $r$ are inside the scaling window and gives a weak bound for arbitrary $n$ and $r$. Both estimates are simple consequences of the results of \cite{MR2551766,sapozhnikov2010upper}, which establish the analogous bounds for \emph{critical} percolation.
Throughout the paper it will be important to establish bounds that hold not just for $G$ but also for \emph{arbitrary subgraphs} of $G$. This is done to circumvent non-monotonicity issues in inductive analyses of percolation, with similar arguments first appearing in \cite{MR2551766}. (For example, the family of partial differential inequalities on generating functions established in \cref{subsec:positivetermI} applies not to the generating function associated to $G$ but to the minimimum of the generating functions associated to the subgraphs of $G$, so that our inductive analysis of these generating functions will require uniform control over subgraphs.)

Let $G$ be a countable graph, let $H$ be a subgraph of $G$, let $v$ be a vertex of $H$, and consider Bernoulli bond percolation on $H$. We write $R_v=R_v(H)$ for the intrinsic radius of the cluster of $v$ in $H$, and write $E_v=E_v(H)$ for the number of edges of $H$ that \emph{touch} the cluster of $v$ in $H$, that is, have at least one endpoint in the cluster of $v$ in $H$. We write $a \vee b := \max\{a,b\}$ and $a \wedge b := \min\{a,b\}$. 

\begin{lemma}
\label{lem:inside_window}
Let $G=(V,E)$ be a connected, locally finite, quasi-transitive graph, and let $p_c=p_c(G)$. Suppose that $\nabla_{p_c}<\infty$.
Then there exists a positive constant $C$ such that the bounds
\[\bP_p^H(R_v \geq r) \leq C \left(\frac{1}{r} \vee (p-p_c) \right)
\qquad \text{ and } \qquad
\bP_p^H(E_v \geq n) \leq C \left(\frac{1}{n^{1/2}} \vee (p-p_c) \right)
\] 
hold for every $r,n\geq 1$, every $p\in [0,1]$, every subgraph $H$ of $G$, and every vertex $v$ of $H$.
\end{lemma}

We stress that in the statement and proof of this estimate, $p_c$ always refers to $p_c(G)$.
The proof will make use of
\emph{Russo's formula} \cite[Theorem 2.32]{grimmett2010percolation}, which states that if $X:\{0,1\}^{E} \to \R$ depends on at most finitely many edges then $\bE_p\left[ X(\omega)\right]$ is a polynomial in $p$ with derivative
\[
\frac{d}{dp}\bE_p\left[ X(\omega)\right] = \sum_{e\in E} \bE_p\left[ X(\omega^e)-X(\omega_e)\right] = \frac{1}{p}\sum_{e\in E} \bE_p\left[ \mathbbm{1}(\omega(e)=1) \left(X(\omega)-X(\omega_e)\right)\right]
\]
for every $p\in (0,1]$, where we let $\omega^e = \omega \cup \{e\}$ and $\omega_e = \omega \setminus \{e\}$.  
 We write $B_\mathrm{int}(v,n)$ for the intrinsic ball of radius $n$ around $v$ in $K_v$, and write $\partial B_\mathrm{int}(v,n) = B_\mathrm{int}(v,n) \setminus B_\mathrm{int}(v,n-1)$ for the set of vertices at intrinsic distance exactly $n$ from $v$.

\begin{proof}[Proof of \cref{lem:inside_window}]
Fix $H$ and $v$ and for each $r\geq 0$ let $B_\mathrm{int}(v,r)$ be the intrinsic ball of radius $r$ around $v$ in $H$, i.e., the set of vertices that can be reached from $v$ by open paths of length at most $r$ in $H$.
We know by the results of \cite{MR2551766,sapozhnikov2010upper} (see also \cite[Section 6]{Hutchcroftnonunimodularperc}) that 
\[
\bP^H_{p}(R_v \geq r) \preceq r^{-1}
\qquad
 \text{ and }
 \qquad
  \bE_{p}^H\left[ \# B_\mathrm{int}(v,r)\right]  \preceq r
\]
for every $0\leq p\leq p_c$.
Observe that, for each $r\geq 1$, if $K_v$ has intrinsic radius at least $r$ and $e$ is such that $K_v(\omega_e)$ does not have intrinsic radius at least $r$, then $e$ must lie on every intrinsic geodesic of length $r$ starting at $v$ in $K_v$. There are clearly at most $r$ such edges, and
  it follows from Russo's formula that
\[
\frac{d}{dp} \bP^H_{p}(R_v \geq r) \leq \frac{r}{p} \bP^H_{p}(R_v \geq r) \leq \frac{r}{p_c} \bP^H_{p}(R_v \geq r)
\]
for every $p_c \leq p \leq 1$ and $r\geq 1$.
 This inequality may be written equivalently as 
\[
\frac{d}{dp} \log \bP^H_{p}(R_v \geq r) \leq \frac{r}{p_c}.
\]
 Integrating this bound between $p$ and $p_c$ yields that
\begin{equation}
\label{eq:offcrit_onearm_simple}
\bP^H_{p}(R_v \geq r) \leq \bP^H_{p_c}(R_v \geq r) \exp\left[\frac{(p-p_c) r}{p_c}\right] \preceq \frac{1}{r} \exp\left[\frac{(p-p_c) r}{p_c}\right]
\end{equation}
for every $p_c \leq p \leq 1$ and $r\geq 1$. Since $\bP^H_{p}(R_v \geq r)$ is decreasing in $r$, it follows that 
\[
\bP^H_{p}(R_v \geq r) \preceq \min\left\{ \frac{1}{\ell} \exp\left[\frac{(p-p_c)}{p_c} \ell\right] : 1 \leq\ell \leq r\right\}
\]
for every $p_c \leq p \leq 1$ and $r\geq 1$. The claimed bound on the tail of the intrinsic radius follows by taking $\ell= r \wedge \lceil (p-p_c)^{-1}\rceil$.

Now, a similar argument to above yields that
\[
\frac{d}{dp} \log \bE_p^H\left[\# B_\mathrm{int}(v,r)\right] \leq \frac{r}{p_c}
\]
for every $p_c \leq p \leq 1$ and $r\geq 1$, and hence that
\begin{equation}
\label{eq:ballwindow}
\bE_p^H\left[\# B_\mathrm{int}(v,r)\right] \preceq r \exp\left[\frac{p-p_c}{p_c}r \right]
\end{equation}
for every  $p_c \leq p \leq 1$ and $r\geq 1$. It follows by the union bound and Markov's inequality that 
\[
\bP_p^H(E_v \geq n) \leq \frac{1}{n}\bE_p^H\left[\# B_\mathrm{int}(v,r)\right] + \bP_p^H(R_v \geq r) \preceq \frac{r}{n} \exp\left[\frac{p-p_c}{p_c}r \right] +  \left[\frac{1}{r}\vee (p-p_c)\right]
\]
for every $n,r \geq 1$. The claim follows by taking $r= \lceil n^{1/2} \wedge (p-p_c)^{-1} \rceil$.
\end{proof}

\section{Upper bounds outside the scaling window}
\label{sec:outside_window}

In this section we prove the upper bounds of \cref{thm:main_radius,thm:main_volume} in the case $p > p_c$. 

\subsection{Setting up the main differential inequalities}
\label{subsec:setup}

Most of the the work to prove \cref{thm:main_radius,thm:main_volume} will concern the case that $p>p_c$ is slightly supercritical and $n$ and $r$ are \emph{outside} the scaling window, so that either $n \gg |p-p_c|^{-2}$  or $r \gg |p-p_c|^{-1}$.                          
As discussed above, we  follow the basic strategy of \cite{HermonHutchcroftSupercritical}, but apply the assumption that $p_c<p_{2\to 2}$ to make the proof quantitative.
We begin by recalling some notation from \cite{HermonHutchcroftSupercritical}.
Let $G=(V,E)$ be a connected, locally finite, transitive, nonamenable graph, and let $v$ be a vertex of $G$. Let $K_v$ denote the cluster of $v$, and let $E_v = |E(K_v)|$ be the number of edges touching $K_v$. Define $\sH$ to be the set of all finite connected subgraphs of $G$, and let $\sH_v$ be the set of all finite connected subgraphs of $G$ containing $v$. Given a function $F: \sH_v \to \R$, we write
\[
\bE_{p,n}[F(K_v)] := \bE_p\left[ F(K_v)\mathbbm{1}(E_v \leq n)\right] \qquad \text{ and } \qquad \bE_{p,\infty}[F(K_v)] := \bE_p\left[ F(K_v)\mathbbm{1}(E_v < \infty)\right] 
\]
for every $p\in [0,1]$ and $n \geq 1$. 



 Given $F: \mathscr{H}_v \to \R$ and $n\geq 1$, Russo's formula allows us to express the derivative of the truncated expectation $\bE_{p,n}[F(K_v)]$, which is a polynomial in $p$, in terms of pivotal edges and obtain that
\begin{equation}
\frac{d}{dp}\bE_{p,n} \left[F(K_v)\right] = \mathbf{U}_{p,n}[F(K_v)]-\mathbf{D}_{p,n}[F(K_v)]
\end{equation}
where we write
\begin{align*}
\mathbf{U}_{p,n}\left[F(K_v)\right] &:= 
 \frac{1}{p} \sum_{e\in E} \bE_{p,n}\left[ \left(F\left[K_v\right]-F\left[K_v(\omega_e)\right]\right) \mathbbm{1}\bigl(\omega(e)=1\bigr) \right]
\intertext{and}
\mathbf{D}_{p,n}\left[F(K_v)\right] &:=  \frac{1}{1-p}\sum_{e\in E} \bE_p\left[ F(K_v) \mathbbm{1}\bigl(\omega(e)=0,\, E_v \leq n < E_v(\omega^e)\bigr) \right].
\end{align*}
See \cite[Section 2]{HermonHutchcroftSupercritical} for further details. Intuitively, in the $n\to \infty$ limit, the term $\mathbf{D}_{p,\infty}\left[F(K_v)\right]$ accounts for the effect of finite clusters becoming infinite, while the term $\mathbf{U}_{p,\infty}\left[F(K_v)\right]$ accounts for the effect of finite clusters growing while remaining finite. (Note however that the above formulas are only \emph{a priori} valid for finite $n$.)
Note that 
 $\bU_{p,n}[F(K_v)]$ is non-negative if $F$ is increasing and that $\bD_{p,n}[F(K_v)]$ is non-negative if $F$ is non-negative.  Note also that  $\bU_{p,n}[F(K_v)]$  and $\bD_{p,n}[F(K_v)]$ both depend linearly on the function $F$.

In order to prove \cref{thm:main_radius,thm:main_volume}, we will need to prove lower bounds on $\bD_{p,n}[F(K_v)]$ and upper bounds on $\bU_{p,n}[F(K_v)]$ for appropriate choices of $F$. The two quantities will often have roughly the same order, making the analysis of their difference  rather delicate.

\subsection{Bounding the negative term}
\label{subsec:negative_term}

In this section we prove a lower bound on $\bD_{p,n}[F(K_v)]$ for non-negative $F$.
In \cite[Proposition 2.1]{HermonHutchcroftSupercritical}, it is shown via an ineffective argument that if $G$ is transitive and nonamenable then for every $p_c<p_0 \leq 1$ there exists a positive constant $c_{p_0}$ such that
\[
\mathbf{D}_{p,n}\left[F(K_v)\right] \geq c_{p_0} \bE_{p,n}\left[ |K_v| \cdot F(K_v)\right]
\]
for every $p_0 \leq p \leq 1$ and every increasing function $F: \sH_v \to [0,\infty)$. A key ingredient to the proof of our main theorems is the following proposition, which allows us to take $c_{p_0}$ of order $(p_0-p_c)$ under the assumption that $p_c<p_{2\to 2}$. We write $\theta_*(p) = \inf_{v\in V} \bP_p(v \to \infty)$ and $\theta^*(p)=\sup_{v\in V} \bP_p(v \to \infty)$ and write $M$ for the maximum degree of $G$.

\begin{prop}
\label{prop:NegativeTermQuant}
Let $G$ be a countable graph. Then 
\begin{equation}
\label{eq:Dquantitative}
\mathbf{D}_{p,n}\left[F(K_v)\right] \geq \left[\frac{\theta_*(p)}{p^2(1-p)M \theta^*(p)\|T_p\|_{2\to 2}^2}\right] \theta_*(p) \bE_{p,n}\left[ |K_v| \cdot F(K_v)\right]
\end{equation}
for every non-negative function $F: \mathscr{H}_v \to [0,\infty)$,   every $n\geq 1$, and every $p\in [0,1)$. Consequently, if $G$ is connected, locally finite, and quasi-transitive with $p_c(G)<p_{2\to 2}(G)$, then there exist positive constants $\delta>0$ and $c>0$ such that
\begin{equation}
\label{eq:Dapprox}
\mathbf{D}_{p,n}\left[F(K_v)\right] \geq c(p-p_c) \bE_{p,n}\left[ |K_v| \cdot F(K_v)\right]
\end{equation}
for every non-negative function $F: \mathscr{H}_v \to [0,\infty)$, every $n\geq 1$, and every $p \in (p_c,p_c+\delta)$.
\end{prop}

The precise form of the argument given below was suggested to us by Antoine  Godin; a similar argument will appear in his forthcoming PhD thesis \cite{Antoine}. We thank him for sharing this argument with us, which substantially simplified our proof. 

\begin{remark}
Note that this is the only stage in our argument in which the $L^2$-boundedness condition (as opposed to the triangle condition) is used directly.
\end{remark}

The proof makes use of the notion of the BK inequality and the associated notion of the \emph{disjoint occurence} $A \circ B$ of two events $A$ and $B$; We refer the unfamiliar reader to \cite[Chapter 2.3]{grimmett2010percolation} for background.

\begin{proof}[Proof of \cref{prop:NegativeTermQuant}]
Let $G$ be a countable graph. For each vertex $v$ of $G$, let $E^\rightarrow_v$ denote the set of oriented edges $e$ of $G$ with $e^-=v$. 
We first claim that for each deterministic finite set of vertices $S \subseteq V$ we have that
\begin{equation}
\label{eq:fixedS1}
\Psi_p(S):=\frac{1}{1-p}\sum_{v \in S} \sum_{e \in E^\rightarrow_v} \mathbbm{1}(e^+ \notin S)\bP_p(e^+ \to \infty \text{ off $S$}) \\ \geq \left[\frac{\theta_*(p)}{p^2(1-p)M \theta^*(p)\|T_p\|_{2\to 2}^2}\right] \theta_*(p)|S|
\end{equation}
for every $0<p<p_{2\to 2}$. (Note that we have written the expression on the right in this way as the bracketed term is of constant order in cases of interest.) The deduction of \eqref{eq:Dquantitative} from \eqref{eq:fixedS1} is identical to the proof of \cite[Proposition 2.1]{HermonHutchcroftSupercritical} and is omitted.
Indeed, the proof of \cite[Proposition 2.1]{HermonHutchcroftSupercritical} shows more generally that
\[
\bD_{p,n} \left[ F(K_v) \right] \geq \bE_{p,n} \left[ F(K_v) \Psi_p(K_v)\right]
\]
for every $p\in [0,1)$, $n\geq 1$, and every non-negative $F:\sH_v \to [0,\infty)$.

Let $S$ be a deterministic finite set of vertices. 
 Let $\partial_E^\rightarrow S$ denote the set of oriented edges of $G$ with $e^- \in S$ and $e^+ \notin S$. Observe that for each $u\in S$ we have that
\[
\{|K_u| = \infty\} \subseteq  \bigcup_{e \in \partial_E^\rightarrow S} \{u \leftrightarrow e^-\} \circ \{ e \text{ open} \} 
\circ
\{ e^+ \to \infty \text{ off $S$}\}.
\]
Indeed, suppose that $u\in S$ is in an infinite cluster, and let $\gamma$ be an infinite simple open path starting at $u$. Since $S$ is finite, there is some last vertex $v$ of $S$ that is visited by $\gamma$. Let $e$ be the edge of $\partial_E^\rightarrow S$ that is crossed by $\gamma$ as it leaves $v$, which is necessarily open. Then the pieces of $\gamma$ before and after crossing $e$ are disjoint witnesses for the events $\{u \leftrightarrow e^-\}$ and $\{e^+ \to \infty \text{ off $S$}\}$, both of which are disjoint from the edge $e$.
Thus, applying the BK inequality and the union bound yields that
\[
\bP_p(u\to \infty) \leq p \sum_{v\in S} T_p(u,v) \sum_{e \in E^\rightarrow_v} \mathbbm{1}(e^+ \notin S) \bP_p(e^+ \to \infty \text{ off $S$})
\]
for every $u \in S$, where we write ``$e^+ \to \infty$ off $S$" to mean that there is an infinite open path starting at $e^+$ that does not visit any vertex of $S$.  Summing over $u$ we obtain that
\begin{equation}
\label{eq:Ssummed}
|S|\theta_*(p) \leq p \sum_{v\in S} \sum_{u \in S} T_p(u,v) \sum_{e \in E^\rightarrow_v} \mathbbm{1}(e^+ \notin S) \bP_p(e^+ \to \infty \text{ off $S$}).
\end{equation}
Define $f: V \to \R$ by
\[f_p(v)= \frac{1}{1-p}\mathbbm{1}(v\in S) \sum_{e \in E^\rightarrow_v} \mathbbm{1}(e^+ \notin S)\bP_p(e^+ \to \infty \text{ off $S$}).\]
Rewriting the above inequality \eqref{eq:Ssummed} in terms of $f$ and applying Cauchy-Schwarz, we obtain that
\[
\frac{\theta_*(p)|S|}{p(1-p)} \leq \langle T_p \mathbbm{1}_S, f \rangle \leq \|T_p\|_{2\to 2} \|\mathbbm{1}_S\|_2 \|f\|_2
\leq \|T_p\|_{2\to 2} |S|^{1/2} \|f\|_1^{1/2} \|f\|_\infty^{1/2},
\]
and since we clearly have that $\|f\|_\infty \leq M \theta^*(p)/(1-p)$, it follows that
\[
\frac{1}{1-p}\sum_{v \in S} \sum_{e \in E^\rightarrow_v} \mathbbm{1}(e^+ \notin S)\bP_p(e^+ \to \infty \text{ off $S$})=\|f\|_1 \geq \left[\frac{\theta_*(p)}{p^2(1-p)M \theta^*(p)\|T_p\|_{2\to 2}^2}\right] \theta_*(p)|S|
\]
as claimed. 

The deduction of \eqref{eq:Dapprox} from \eqref{eq:Dquantitative} follows by standard arguments: Indeed, if $G$ is connected and quasi-transitive then there exists $C$ such that $\theta_*(p)\geq p^C\theta^*(p)$ for every $p\in [0,1]$, while if $p_c(G)<p_{2\to 2}(G)$ then $\|T_p\|_{2\to 2}$ is bounded on a neighbourhood of $p_c$. On the other hand, for quasi-transitive graphs there always exists a positive constant $c$ such that $\theta_*(p) \geq c (p-p_c)$ for all $p_c \leq p \leq 1$ \cite{duminil2015new}. Together these observations allow us to deduce \eqref{eq:Dapprox} from \eqref{eq:Dquantitative}.
\end{proof}

\begin{remark}
The proof of \cite[Proposition 2.1]{HermonHutchcroftSupercritical} can also be made quantitative under the assumption that $p_c<p_{2\to 2}$, since in this case we know that the density of trifurcations is of order $(p-p_c)^3$ \cite[Corollary 5.6]{hutchcroft20192}. Note however that the resulting bound is not sharp.
\end{remark}

An easy corollary of \cref{prop:NegativeTermQuant} is the following  weak version of the first moment estimate from \cref{cor:moments}. This weak estimate will nevertheless be useful to us as boundary data when we analyze a certain differential inequality later in the paper.

\begin{corollary}
\label{cor:good_points}
Let $G=(V,E)$ be a connected, locally finite, quasi-transitive graph such that $p_c<p_{2\to 2}$. Then there exist positive constants $\delta$ and $C$ such that
\[
\inf\Bigl\{ (p-p_c)\bE_{p,\infty} |K_v| : p \in (p_c+\eps,p_c+2 \eps)\Bigr\} \leq C
\]
for every $v\in V$ and $0<\eps \leq \delta$
\end{corollary}

\begin{proof}
Fix $v\in V$. Since $G$ is quasi-transitive and satisfies the triangle condition, there exists a constant $C$ such that $\bP_p( |K_v|=\infty) \leq C (p-p_c)$ for every  $p_c \leq p \leq 1$ \cite{MR1127713} (this also follows from \cref{lem:inside_window}). On the other hand, \cref{prop:NegativeTermQuant} implies that there exist positive constants $c$ and $\delta$ such that
\[
\frac{d}{dp} \bP_p(|K_v|> n)= - \frac{d}{dp} \bE_{p,n}[1]=\bD_{p,n}[1] \geq c(p-p_c) \bE_{p,n} |K_v|
\]
for every $n\geq 1$ and $p\in (p_c,p_c+\delta]$.
Integrating this differential inequality yields that
\[\int_{p_c+\eps}^{p_c+2\eps} c(p-p_c) \bE_{p,n} |K_v| \dif p\leq \bP_{p_c+2\eps}(|K_v|> n) - \bP_{p_c+\eps}(|K_v|>n)
 \leq \bP_{p_c+2\eps}(|K_v|> n)
\]
for every $0<\eps \leq \delta/2$. Using the monotone convergence theorem to take the limit as $n\to \infty$, we obtain that
\[\int_{p_c+\eps}^{p_c+2\eps} c(p-p_c) \bE_{p,\infty} |K_v| \dif p
 \leq  \bP_{p_c+2\eps}(|K_v|=\infty) \leq 2C \eps
\]
for every $0<\eps \leq \delta/2$, where the final inequality follows from \cref{lem:inside_window}. This is easily seen to imply the claim.
\end{proof}

\subsection{Skinny clusters and the intrinsic radius}
\label{subsec:positive_term}

The goal of this section is to prove the following proposition, which establishes the upper bounds of \cref{thm:main_radius} in the slightly supercritical regime. This is substantially easier than the corresponding upper bounds on the tail of the volume.

\begin{prop}
\label{prop:int_rad_supercritical_upper}
Let $G=(V,E)$ be a connected, locally finite, quasi-transitive graph such that $p_c<p_{2\to 2}$. Then there exist positive constants $\delta$, $c$, and $C$ such that
\[\bP_p(r \leq \operatorname{Rad}_\mathrm{ext}(K_v) < \infty) \leq \bP_p(r \leq \operatorname{Rad}_\mathrm{int}(K_v) < \infty) \leq C r^{-1} e^{-c(p-p_c)r}\]
for every $r\geq 1$, $v\in V$, and $p\in [p_c,p_c+\delta)$.
\end{prop}

We begin with the following proposition, 
which upper bounds the probability of having a large \emph{skinny} cluster, whose radius is large but whose volume is smaller than it should be given the large radius.
In particular, this proposition applies the assumption $\nabla_{p_c}<\infty$ to give a quantitative improvement to \cite[Lemma 2.8]{HermonHutchcroftSupercritical}. This proposition will be extremely useful to us, and will be applied many times throughout the paper. Again, it will be important for these future applications to have bounds that hold in arbitrary subgraphs of our fixed transitive graph $G$.


\begin{prop}[Skinny clusters]
\label{lem:SkinnyRadius}
Let $G=(V,E)$ be a connected, locally finite, quasi-transitive graph such that $\nabla_{p_c}<\infty$. 
 There exist positive constants $\delta$, $c$, and $C$ such that the bound
\begin{equation*}
\bP^H_p(r \leq R_v < \infty \text{ and } E_v \leq \alpha R_v) \leq C \inf \left\{ \left(\frac{1}{r} + \lambda\right) \exp\left[- c e^{-C \lambda \alpha} \lambda r\right] : 0 \vee (p-p_c) \leq \lambda \leq \delta \right\}
\end{equation*}
 holds for every $0 \leq p \leq p_c+\delta$, $\alpha\geq 1$, $r \geq 0$, subgraph $H$ of $G$, and vertex $v$ of $H$. 
\end{prop}


Here, we recall that $E_v$ denotes the number of edges of $H$ touched by the percolation cluster of $v$ in $H$, and $R_v$ denotes the intrinsic radius of this cluster. 
Again, we stress that in the statement and proof of this proposition, $p_c$ will always denote $p_c(G)$ and all implicit constants will depend only on $G$.

\begin{proof}[Proof of \cref{lem:SkinnyRadius}]
 Fix a subgraph $H$ of $G$, a vertex $v$ of $H$, $0 \leq p \leq 1$, $r \geq 1$ and $\alpha \geq 1$. 
Let $\lambda \geq (p-p_c) \vee 0$. 
It suffices to prove that there exist positive constants $\delta$, $c$ and $C$ depending only on $G$ such that if $\lambda \leq \delta$ then
\begin{equation}
\label{eq:SkinnyReStatement0}
\bP^H_p(r \leq R_v < \infty \text{ and } E_v \leq \alpha R_v) \preceq \left(\frac{1}{r} + \lambda\right) \exp\left[- c e^{-C \lambda \alpha} \lambda r\right].
\end{equation}
Let $n=\lceil 1/\lambda \rceil+2$. 
The case $r = O(n)$ of \eqref{eq:SkinnyReStatement0} may be deduced easily from \cref{lem:inside_window}: Indeed, it follows from \cref{lem:inside_window}  that
 there exists $\delta_1 > 0 $ such that if $0\leq p \leq p_c+\delta_1$ and $r \leq 4n$ then 
\begin{equation}
\label{eq:Skinny_small_n}
\bP^H_p(r \leq R_v < \infty \text{ and } E_v \leq \alpha R_v) \leq \bP_p^{H}(R_v \geq r) \preceq \frac{1}{r}.
\end{equation}
The bound \eqref{eq:Skinny_small_n} is already of the desired order when $r \leq 4n$, since the quantity in the exponential on the right hand side of \eqref{eq:SkinnyReStatement0} is bounded in this regime. Thus, it suffices to prove that there exist positive constants $\delta_2$, $c$, and $C$ depending only on $G$ such that if $\lambda \leq \delta_2$ then
\begin{equation}
\label{eq:SkinnyReStatement}
\bP^H_p(r \leq R_v < \infty \text{ and } E_v \leq \alpha R_v) \preceq \lambda \exp\left[- c e^{-C \lambda \alpha} \lambda r\right]
\end{equation}
for every $r \geq 4n$.

\begin{figure}
\centering
\includegraphics{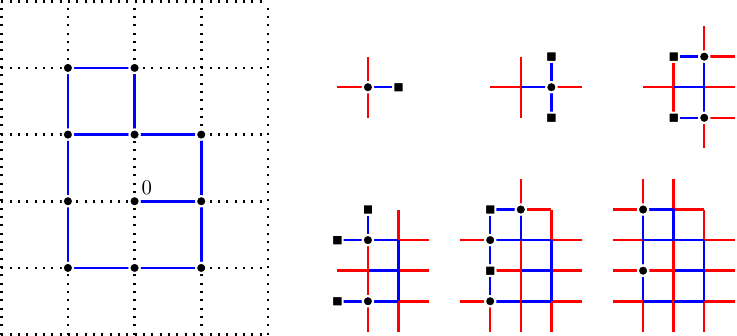}
\caption{Exploring a cluster in a breadth-first manner. Left: A possible percolation cluster in $\Z^2$ with intrinsic radius $5$. Solid blue edges are open, dotted edges are closed. Right: Exploring the cluster in a breadth-first manner starting from the origin. At stage $i$, we reveal the status of all edges that are incident to the set of vertices of intrinsic distance exactly $i$ from the origin (black discs) and that have not already been revealed. Revealed open edges are blue and revealed closed edges are red. Running the exploration for $i$ steps reveals the intrinsic ball of radius $i$, with vertices of distance exactly $i$ represented by black squares. In this instance the process terminates after the cluster fails to grow on the sixth step.}
\label{fig:exploration}
\end{figure}

To this end, suppose that $r\geq 4n$, so that $k:=\lfloor r/2n \rfloor-1 \geq 1$. Suppose further that $0\leq p \leq p_c+\delta_1$ and that $0\leq \lambda \leq \delta_1$. 
Consider exploring the cluster of $v$ as follows: at stage $i$,  expose the value of those edges that touch $\partial B_{\mathrm{int}}(v,i-1)$, the set of vertices with intrinsic distance exactly $i-1$ from $v$, and have not yet been exposed. Stop when $\partial B_{\mathrm{int}}(v,i)=\eset $. Here, by \emph{intrinsic distance} we always mean the graph distance \emph{on the percolation cluster} (which is a subgraph of $H$). See \cref{fig:exploration} for an illustration of this exploration process.
For each $i \geq 0$ let $X_i$ be the set of edges whose status is queried at stage $i+1$, so that $X_i$ is determined by the process run up to time $i$ and $|X_i| >0$ for every $0 \leq i \leq r-1$ on the event that $R_v \geq r$.
Define a sequence of stopping times $(T_j)_{j\geq 0}$ for this exploration process by setting $T_0=0$ and recursively  setting 
\[T_{j+1} = \inf\Bigl\{i \geq T_j + n : 0<|X_i| \leq 4\alpha \Bigr\},
\]
letting $T_{j+1}=\infty$ if the set on the right hand side is empty.
Recalling that $k:=\lfloor r/2n \rfloor-1$, we claim that $T_{k}<\infty$ on the event that $r \leq R_v <\infty$ and $E_v \leq \alpha R_v$. Indeed, suppose that this event holds. Let $k'=k'(K_v)=\lfloor R_v/2n\rfloor-1$, so that $k'\geq k \geq 1$ and $k' \geq R_v/4n$. We trivially have that 
$2nk'+n-1 = 2n(\lfloor R_v/2n\rfloor-1)+n-1 \leq R_v-1$
 and that
\[
\sum_{a=1}^{2k'}\sum_{b=0}^{n-1} |X_{an+b}| \leq \sum_{i=n}^{R_v-1} |X_i| \leq E_v \leq \alpha R_v,
\]
and it follows that there exists $0\leq b = b(K_v) \leq n-1$ such that
$\sum_{a=1}^{2k'} |X_{an+b}| \leq \alpha R_v/n$.
Applying Markov's inequality, we deduce that there exists a subset $A=A(K_v)$ of $\{1,\ldots,2k'\}$ such that $|A| \geq k'$ and $|X_{an+b(K_v)}| \leq \alpha R_v /n k' \leq 4\alpha$ for every $a \in A$. If we enumerate $A$ in increasing order as $A=\{a_1,a_2,\ldots\}$, then an easy induction shows that $T_i \leq a_i n + b < \infty$ for every $i \leq k'$ and hence for every $i\leq k$ as claimed.

\begin{figure}
\centering
\includegraphics[width=0.8\textwidth]{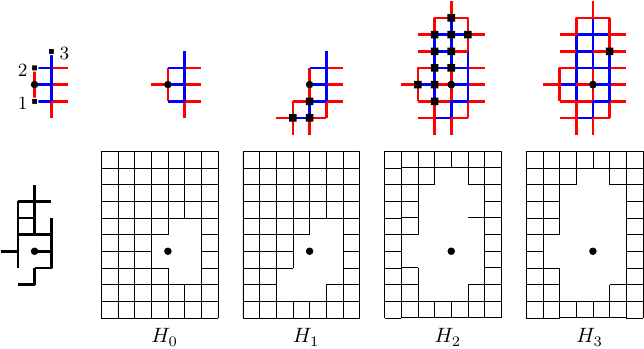}
\caption{Sequentially exploring the clusters of the boundary edges $X_{T_i}$. Left: A cluster in $\Z^2$ (bottom left) and the three vertices in the set $Y_{3}$ (top left), enumerated clockwise. Right: The subgraphs $H_0,\ldots,H_3$ determined by sequentially exploring the clusters of the edges $v_1,\ldots,v_3$ in the complement of the explored region $H_0$. At each step, revealed open edges are blue, revealed closed edges are red, and the vertices of $K_i$ are marked with black squares.}
\end{figure}

Let $\cF_i$ be the $\sigma$-algebra generated by the first $i$ steps of the exploration process, and let $\cF_{T_i}$ be the stopped $\sigma$-algebra associated to the stopping time $T_i$. We clearly have that
\[
\bP^H_p(T_1 < \infty \mid \cF_{T_0}) = \bP^H_p(T_1 < \infty) \leq \bP^H_p(R_v \geq n) \preceq \lambda,
\]
where the final inequality follows from \eqref{eq:Skinny_small_n}. Now let $i\geq 1$ and condition on $\cF_{T_i}$. If $T_i=\infty$ then we trivially have that $T_{i+1}=\infty$ also. Now suppose that $T_i<\infty$. Let $Y_{T_i}$ be the set of vertices of $\partial B_\mathrm{int}(v,T_i)$ that have an edge of $X_{T_i}$ incident to them, so that $|Y_{T_i}|\leq 2 |X_{T_i}|$.
Enumerate the edges of $Y_{T_i}$ by $Y_{T_i} = \{w_1,\ldots,w_\ell\}$.  Let $H_0$ be the subgraph of $H$ spanned those edges that have \emph{not} been queried by time $T_i$ (i.e., those edges not in $\bigcup_{j=0}^{T_i-1} X_j$). Let $K_1$ be the cluster of $w_1$ in $H_0$ and let $H_1$ be the subgraph of $H_0$ defined by deleting every edge that touches $K_1$ from $H_0$. (Recall that we say an edge of $H_0$ \emph{touches} $K_1$ if it has at least one endpoint in the vertex set of $K_1$.) Inductively, for each $2 \leq j \leq \ell$, let $K_j$ be the cluster of $w_j$ in $H_{j-1}$ and let $H_j$ be the subgraph of $H_{j-1}$ formed by deleting every edge that touches $K_j$ from $H_{j-1}$. 
 In order for $T_{i+1}$ to be finite, we must have that there exists $1 \leq j \leq \ell$ such that $K_j$ is non-empty and has intrinsic radius at least $n-1$ from the perspective of $w_j$. 
 It follows from \cref{lem:inside_window} applied to the subgraph $H_{j-1}$ that there exists a constant $C$  such that for each $1 \leq j \leq \ell$, the conditional probability that $K_j$ has intrinsic radius at least $n-1$ given $\cF_{T_i}$ and the clusters $K_1,\ldots,K_{j-1}$ is at most $C \lambda$, and hence that
\begin{align*}
\bP^H_p(T_{i+1} < \infty \mid \cF_{T_i}) 
 &\leq \mathbbm{1}(T_i < \infty) \left[1-\left|1-1 \wedge C \lambda \right|^{Y_{T_i}}\right] \leq
 \mathbbm{1}(T_i < \infty) \left[1-\left|1-1 \wedge C \lambda \right|^{8\alpha}\right].
\end{align*}
(Note that at this stage it was very important that the bounds of \cref{lem:inside_window} held for arbitrary subgraphs of $H$, not just for $H$ itself.) Taking products and using the bound $1-x \leq e^{-x}$, we obtain that
\[
\bP^H_p(T_i<\infty) \preceq \lambda \left[1-|1-1\wedge C\lambda|^{8\alpha}\right]^{i-1}\leq 
\lambda \exp\left[-|1-1\wedge C\lambda|^{8\alpha}(i-1)\right]
\]
for every $i\geq 1$ and hence that
\[
\bP^H_p(R_v \geq r, E_v \leq \alpha R_v) \leq 
\bP^H_p(T_k<\infty)
\preceq \lambda \exp\left[-|1-1\wedge C \lambda|^{8\alpha}(k-1)\right].
\]
Now, since $1-x \geq e^{-2x}$ for small non-negative values of $x$, it follows that there exist positive constants $\delta_2$, $c$ and $C'$ such that if $\lambda \leq \delta_2$ and $r \geq 4n$ then
\begin{equation}
\label{eq:skinny_large_n}
\bP^H_p(R_v \geq r, E_v \leq \alpha R_v) 
\preceq \lambda \exp\left[-ce^{-C'\lambda \alpha}\lambda r\right].
\end{equation}
The proof may be concluded by combining the bounds \eqref{eq:Skinny_small_n} and \eqref{eq:skinny_large_n}, which hold for $r \leq 4n$ and $r\geq 4n$ respectively.
\end{proof}

\begin{remark}
The expression
$e^{-C\lambda \alpha}\lambda$ is maximized by $\lambda = 1/C\alpha$. In particular, taking $\alpha = rs$ and $\lambda = 1/rs$, it follows from \cref{lem:SkinnyRadius} and \cref{lem:subcritical_radius_lower} that, under the hypotheses of those results, there exist constants $c$ and $C$ such that
\begin{equation}
\label{eq:stretched_exponential_remark}
\bP_{p_c}\bigl(E_v \leq s^{-1} r^{2} \mid R_v \geq r \bigr) \leq C e^{-c s}
\end{equation}
for every $v\in V$, $r\geq 1$, and $s \geq 1$. 
\end{remark}

We now apply \cref{prop:NegativeTermQuant,lem:SkinnyRadius} to prove \cref{prop:int_rad_supercritical_upper}.

\begin{proof}[Proof of \cref{prop:int_rad_supercritical_upper}]
The first inequality is trivial, so it suffices to prove the second. 
It follows from \cref{prop:NegativeTermQuant} that there exist positive constants $\delta_1$ and $c_1$ such that
\begin{align*}
\frac{d}{dp}\bP_{p,n}(R_v \geq r) &\leq - c_1 (p-p_c)\bE_{p,n}\left[E_v \mathbbm{1}(R_v \geq r)\right] + \bU_{p,n}\left[\mathbbm{1}(R_v \geq r)\right]
\end{align*}
for every $r\geq 1$, $n\geq 1$ and $p\in [p_c,p_c+\delta_1)$.
As in the proof of \cref{lem:inside_window}, we can bound $p \bU_{p,n}\left[\mathbbm{1}(R_v \geq r)\right]$ by the expected number of open edges $e$ such that the cluster of $v$ has intrinsic radius at least $r$ in $\omega$ and strictly less than $r$ in $\omega_e$. Since any such open edge must lie on every intrinsic geodesic of length $r$ starting from $v$ in $\omega$, we deduce that
\begin{align}
\label{eq:Rad_diffineq}
\frac{d}{dp}\bP_{p,n}(R_v \geq r)
& \leq - c_1 (p-p_c)\bE_{p,n}\left[E_v \mathbbm{1}(R_v \geq r)\right] + \frac{r}{p_c}\bP_{p,n}(R_v \geq r)
\end{align}
for every $r\geq 1$, $n\geq 1$ and $p\in [p_c,p_c+\delta_1)$. 

On the other hand, it follows from \cref{lem:SkinnyRadius} that there exists positive constants $\delta_2$, $c_2,c_3$, $C_1$, and $C_2$ such that
\begin{align*}
\bP_{p,\infty}\Bigl(R_v \geq r \geq \frac{1}{2}c_1p_c(p-p_c) E_v\Bigr) &\leq C_1\left[\frac{1}{r}+(p-p_c)\right]\exp\left[-c_2 e^{-2C_1 (p-p_c) [c_1p_c(p-p_c)]^{-1}}(p-p_c) r\right]
\\
&\leq 
C_1\left[\frac{1}{r}+(p-p_c)\right]\exp\left[-2c_3(p-p_c) r\right]
\\
&\leq \frac{C_2}{r} \exp\left[-c_3(p-p_c) r\right]
\end{align*}
for every $r\geq 1$ and $p\in [p_c,p_c+\delta_2)$, where we used that
$x e^{-2 x} \leq e^{-x-1}$ for every $x\geq 0$ in the final inequality. 
It follows that 
\begin{align*}
c_1 (p-p_c)\bE_{p,n}\left[E_v \mathbbm{1}(R_v \geq r)\right] &\geq \frac{2r}{p_c} \bP_{p,n}(R_v \geq r, c_1 (p-p_c) E_v \geq 2r)
\\
&\geq 
\frac{2r}{p_c} \bP_{p,n}(R_v \geq r) - \frac{2r}{p_c}\bP_{p,\infty}\Bigl(R_v \geq r \geq \frac{1}{2}c_1p_c(p-p_c) E_v\Bigr)\\
&\geq 
\frac{2r}{p_c} \bP_{p,n}(R_v \geq r) - \frac{2C_2}{p_c} \exp\left[-c_3 (p-p_c)r\right]
\end{align*}
for every $r\geq 1$ and $p\in [p_c,p_c+\delta_2)$.
Letting $\delta_3=\delta_1 \wedge \delta_2$, we deduce from this and \eqref{eq:Rad_diffineq} that
\begin{align*}
\frac{d}{dp}\bP_{p,n}(R_v \geq r)
& \leq - \frac{r}{p_c}\bP_{p,n}(R_v \geq r) + \frac{2C_2}{p_c} \exp\left[-c_3 (p-p_c)r\right]
\end{align*}
for every $r\geq 1$ and $p\in [p_c,p_c+\delta_3)$. 
Letting $c_4=c_3 \wedge (1/p_c)$ and $C_3=2C_2/p_c$, it follows that
\begin{align*}
\frac{d}{dp}\left[e^{c_4(p-p_c)r}\bP_{p,n}(R_v \geq r)\right]
&\leq
\left[ c_4 r - \frac{r}{p_c}\right] e^{c_4(p-p_c)r}\bP_{p,n}(R_v \geq r) 
+ 
\frac{2C_2}{p_c} e^{(c_4-c_3) (p-p_c)r}\\
&\leq 
  \frac{2C_2}{p_c} e^{(c_4-c_3) (p-p_c)r} \leq C_3.
\end{align*}
Integrating this bound 
yields that there exist constants $C_4$ and $C_5$ such that
\begin{align*}
\bP_{p,n}(R_v \geq r) &\leq \bP_{p_c,n}(R_v\geq r) e^{-c_4(p-p_c)r} + C_3(p-p_c)e^{-c_4(p-p_c)r}
\\
&\leq C_4\left(\frac{1}{r} + (p-p_c)\right)e^{-c_4(p-p_c)r} \leq
\frac{C_5}{r} e^{-c_4(p-p_c)r/2} 
\end{align*}
for every $1 \leq r,n < \infty$ and $p\in [p_c,p_c+\delta_3)$. The claim follows by taking $n\to\infty$.
\end{proof}

\subsection{Bounding the positive term I: derivation of the auxiliary differential inequality}
\label{subsec:positivetermI}

The goal of the following two subsections is to prove the upper bound of \cref{thm:main_volume} in the slightly supercritical regime. This is the most technical part of the paper.

\begin{prop}
\label{prop:vol_supercritical_upper}
Let $G=(V,E)$ be a connected, locally finite, quasi-transitive graph such that $p_c<p_{2\to 2}$. Then there exist positive constants $\delta$, $c$, and $C$ such that
\begin{equation}
\label{eq:vol_supercritical_upper}
\bP_p(n \leq |K_v| < \infty) \leq C n^{-1/2} \exp\left[-c(p-p_c)^2n\right]
\end{equation}
for every $n\geq 1$, $v\in V$, and $p\in (p_c,p_c+\delta)$.
\end{prop}

To prove this proposition, it suffices to prove that there exist positive constants $c$, $C$, and $\delta$ such that
\begin{equation}
\label{eq:STP}
\bE_{p,\infty} \left[|K_v|\exp\left(c(p-p_c)^2 |K_v|\right)\right] \leq \frac{C}{p-p_c}
\end{equation}
for every $p\in (p_c,p_c+\delta)$. Indeed, Markov's inequality will then imply that
\[\bP_p(n \leq |K_v| < \infty) \leq \frac{C}{(p-p_c)n} \exp\left[-c(p-p_c)^2n\right]\]
for every $p\in (p_c,p_c+\delta)$ and $n\geq 1$, 
which is of the correct order when $n\geq (p-p_c)^{-2}$. On the other hand, if $n \leq (p-p_c)^{-2}$ then a bound of the correct order is already provided by \cref{lem:inside_window}.

\medskip

The primary remaining obstacle we must overcome in order to prove \eqref{eq:STP} is to establish upper bounds on $\bU_{p,n}[|K_v|^k]$, the positive part of the derivative of the  truncated $k$th moment. Our approach will follow a similar philosophy to that of \cite[Section 2.3]{HermonHutchcroftSupercritical}. Unfortunately, while the methods developed in that paper are quantitative, they are not sharp, and eventually lead to a factor of order $(p-p_c)^4$ rather than of order $(p-p_c)^2$ in the exponent of \eqref{eq:vol_supercritical_upper} when combined with our sharp control of skinny clusters, \cref{lem:SkinnyRadius}. Obtaining optimal bounds requires a rather more delicate and technical approach. In this subsection, we derive a differential inequality which we will use to bound these quantities; the analysis of this differential inequality is then performed in the next subsection. We refer to this differential inequality as the \emph{auxiliary} differential inequality to distinguish it from the other differential inequalities we have been interested in.

\medskip

As in \cite{HermonHutchcroftSupercritical}, we begin by expressing $\bU_{p,n}[|K_v|^k]$ geometrically in terms of bridges. We first recall the relevant definitions.
 Let $H$ be a connected graph. Recall that two vertices $u$ and $v$ of $H$ are said to be \textbf{$2$-connected} if $u$ and $v$ remain connected when any edge is deleted from $H$. (In particular, every vertex is $2$-connected to itself.) Equivalently, by Menger's theorem, $u$ and $v$ are $2$-connected if there exist a pair of edge-disjoint paths each connecting $u$ to $v$.  This defines an equivalence relation on the vertices of $H$, the pieces of which are referred to as the \textbf{$2$-connected components} of $H$. We write $[v]$ for the $2$-connected component of the vertex $v$ in $H$.
An edge $e$ of $H$ is said to be a \textbf{bridge} of $H$ if the graph formed by deleting $e$ from $H$ is disconnected. Equivalently, $e$ is a bridge of $H$ if its endpoints are in distinct $2$-connected components of $H$. We define $\operatorname{Tr}(H)$ to be the tree whose vertices are the $2$-connected 
components of $H$ and whose edges are the bridges of $H$.
Given a graph $H$ and a sequence of vertices $v_1,\ldots,v_k$ of $H$, let $\Bridges(v_1,\ldots,v_k;H)$ be the number of edges in the subtree of $\operatorname{Tr}(H)$ spanned by the union of the geodesics between the vertices $[v_1],\ldots,[v_k]$ in the tree of $2$-connected components $\operatorname{Tr}(H)$. 

\medskip

Let $G$ be a connected, locally finite, and quasi-transitive, let $p\in [0,1]$ and let $v\in V$. We have by \cref{prop:NegativeTermQuant} that there exist positive constants $c$ and $\delta$ such that if $p_c < p \leq p_c+\delta$ then the $p$-derivative of $\bE_{p,n}\left[|K_v|e^{u|K_v|}\right]$ satisfies
\begin{align}
\partial_p \bE_{p,n}\left[|K_v|e^{u|K_v|}\right] &\leq -c(p-p_c) \bE_{p,n}\left[|K_v|^2e^{u|K_v|} \right] + \bU_{p,n}\left[|K_v|e^{u|K_v|}\right]
\nonumber
\\
&=-c(p-p_c) \bE_{p,n}\left[|K_v|^2e^{u|K_v|} \right] + \sum_{k=0}^\infty \frac{u^k}{k!}\bU_{p,n}\left[|K_v|^{k+1}\right]
\label{eq:pderivative1}
\end{align}
for every $u\geq 0$ and $n \geq 1$. 
Observe that, by definition of the relevant quantities,  we may express $\bU_{p,n}[|K_v|^k]$ as
\begin{align*}
\bU_{p,n}[|K_v|^k] &= \sum_{\substack{x_1,\ldots,x_k \\ \in V(G)}} \bU_{p,n}\left[\mathbbm{1}(x_1,\ldots,x_k \in K_v)\right] 
\\&
= \frac{1}{p} \sum_{\substack{x_1,\ldots,x_k \\ \in V(G)}} \bE_{p,n}\left[\mathbbm{1}(x_1,\ldots,x_k \in K_v)\Bridges(v,x_1,\ldots,x_k;K_v)\right],
\end{align*}
where the second equality follows from the fact that if $x_1,\ldots,x_k$ all belong to $K_v$ then the quantity $\Bridges(v,x_1,\ldots,x_k;K_v)$ is equal to the number of edges that are open pivotals for this event.
Writing $\Bridges(v,x_1,\ldots,x_k;K_v) = \Bridges(v,x_1,\ldots,x_k)$, this can be written more succinctly as
\begin{align*}
\bU_{p,n}[|K_v|^k] &
=\frac{1}{p}  \bE_{p,n}\left[\sum_{\substack{x_1,\ldots,x_k\\ \in K_v}}\Bridges(v,x_1,\ldots,x_k)\right]
\end{align*}
for every $n,k\geq 1$. 
%
%
%
%
Summing over $k$ it follows that
\begin{align*}
&\bU_{p,n}\left[|K_v|e^{u|K_v|}\right]\\ &\hspace{0.2cm}= 
\frac{1}{p} \sum_{k=0}^\infty \frac{u^{k}}{k!} \bE_{p,n}\left[\sum_{\substack{x_1,\ldots,x_k\\ \in K_v}}\Bridges(v,x_1,\ldots,x_{k+1}) \mathbbm{1}\left(\Bridges(v,x_1,\ldots,x_{k+1}) \geq \frac{1}{2}cp(p-p_c)|K_v| \right)\right]
\\&\hspace{0.2cm}\hspace{0.26cm}+\frac{1}{p} \sum_{k=0}^\infty \frac{u^{k}}{k!} \bE_{p,n}\left[\sum_{\substack{x_1,\ldots,x_k\\ \in K_v}}\Bridges(v,x_1,\ldots,x_{k+1}) \mathbbm{1}\left(\Bridges(v,x_1,\ldots,x_{k+1}) < \frac{1}{2}cp(p-p_c)|K_v| \right)\right]
\end{align*}
for every $n \geq 1$ and $u\geq 0$, from which we deduce that
\begin{align}
&\bU_{p,n}\left[|K_v|e^{u|K_v|}\right]
\nonumber
\\
&\hspace{0.2cm}\leq \frac{1}{p} \sum_{k=0}^\infty \frac{u^{k}}{k!} \bE_{p,n}\left[\sum_{\substack{x_1,\ldots,x_k\\ \in K_v}}\Bridges(v,x_1,\ldots,x_{k+1}) \mathbbm{1}\left(\Bridges(v,x_1,\ldots,x_{k+1}) \geq \frac{1}{2}cp(p-p_c)|K_v| \right)\right]
\nonumber
\\
&\hspace{0.2cm}\hspace{0.26cm}+\frac{1}{2}c(p-p_c)\bE_{p,n}\left[|K_v|^2e^{u|K_v|}\right]
\label{eq:U_expansion}
\end{align}
for every $u \geq 0$ and $n\geq 1$. We have split the equation up this way precisely so that the second term can be absorbed into the negative term in \eqref{eq:pderivative1}. Indeed, the inequalities \eqref{eq:pderivative1} and \eqref{eq:U_expansion} together imply that if $G$ is a connected, locally finite, quasi-transitive graph with $p_c<p_{2\to 2}$ then there exist constants $\delta$, $c_1$, and $c_2$ such that
\begin{multline}
\label{eq:diffineq_mgf}
\partial_p \bE_{p,n}\left[|K_v|e^{u|K_v|}\right] \leq -c_1(p-p_c)\bE_{p,n}\left[|K_v|^2e^{u|K_v|}\right] \\
+ \frac{1}{p} \sum_{k=0}^\infty \frac{u^{k}}{k!} \bE_{p,n}\left[\sum_{\substack{x_1,\ldots,x_{k+1}\\ \in K_v}}\Bridges(v,x_1,\ldots,x_{k+1}) \mathbbm{1}\Bigl(\Bridges(v,x_1,\ldots,x_{k+1}) \geq c_2(p-p_c)|K_v| \Bigr)\right]
\end{multline}
for every $v\in V$, $u\geq 0$, $p \in (p_c,p_c+\delta)$ and $1 \leq n < \infty$. Intuitively, the constraint that $\Bridges(v,x_1,\ldots,x_{k+1}) \geq c_2(p-p_c)|K_v|$ can be thought of as a higher-order version of the skinniness constraint which we studied in \cref{lem:SkinnyRadius}. 

\medskip

We will control the summands on the right hand side of \eqref{eq:diffineq_mgf} by an inductive analysis of certain generating functions, which we now introduce.
Let $G$ be a countable, locally finite graph, let $p\in [0,1]$, and let $v$ be a vertex of $G$. For each $k\geq 1$ and $n\in \N_\infty = \{1,2,\ldots\}\cup\{\infty\}$ we define $\sG_{k,n} ( \,\cdot\,,\, \cdot\, ; G,v,p): \R^2 \to [0,\infty]$ by
\begin{multline*}
 \mathscr{G}_{k,n}(s,t;G,v,p) 
=
\sum_{a=0}^\infty \sum_{b=0}^\infty \sum_{\substack{x_1,\ldots,x_k\\ \in V(G)}} \bP_{p,n}^G\Bigl(x_1,\ldots,x_k \in K_v, E_v = a, \Bridges(v,x_1,\ldots,x_k;K_v) =b\Bigr) e^{sa+tb},
\end{multline*}
which is a sort of multivariate generating function, and also define 
\[
\mathscr{F}_{k,n}(s,t;G,p) := \sup\left\{\mathscr{G}_{k,n}(s,t;H,u,p) : H \text{ a subgraph of $G$, $u$ a vertex of $H$} \right\}.
\]
Finally, for each $n\in \N_\infty$ define $\sM_{n} ( \,\cdot\,,\, \cdot\,,\,\cdot\, ; G,v,p): \R^2 \times [0,\infty) \to [0,\infty]$ by
\begin{equation}
\label{eq:Mdef}
\sM_n(s,t,u;G,p) := \sum_{k=0}^\infty \frac{u^k}{k!} \sF_{k+1,n}(s,t;G,p).
\end{equation}
Note that if $G$ is a connected, locally finite, quasi-transitive graph with $p_c<p_{2\to 2}$ and $0\leq s \leq c_2(p-p_c) t$ then we have trivially that  
$\mathbbm{1}\Bigl(\Bridges(v,x_1,\ldots,x_{k+1};K_v) \geq c_2(p-p_c)|K_v| \Bigr) \leq \exp(-s|K_v|+t \Bridges(v,x_1,\ldots,x_{k+1};K_v))$ for every $x_1,\ldots,x_{k+1} \in K_v$ and hence that
the expression appearing on the right hand side of \eqref{eq:diffineq_mgf} can be bounded
\begin{align}
&\frac{1}{p} \sum_{k=0}^\infty \frac{u^{k}}{k!} \bE_{p,n}\left[\sum_{\substack{x_1,\ldots,x_{k+1}\\ \in K_v}}\Bridges(v,x_1,\ldots,x_{k+1};K_v) \mathbbm{1}\Bigl(\Bridges(v,x_1,\ldots,x_{k+1};K_v) \geq c_2(p-p_c)|K_v| \Bigr)\right]
\nonumber\\
&\hspace{2cm}\leq 
\frac{1}{p} \sum_{k=0}^\infty \frac{u^{k}}{k!} \bE_{p,n}\left[\sum_{\substack{x_1,\ldots,x_{k+1}\\ \in K_v}}\Bridges(v,x_1,\ldots,x_{k+1};K_v) e^{-s|K_v|+t \Bridges(v,x_1,\ldots,x_{k+1};K_v)}\right]
\nonumber\\
&\hspace{2cm}\leq 
\frac{e}{tp} \sum_{k=0}^\infty \frac{u^{k}}{k!} \bE_{p,n}\left[\sum_{\substack{x_1,\ldots,x_{k+1}\\ \in K_v}} e^{-s|K_v|+2t \Bridges(v,x_1,\ldots,x_{k+1};K_v)}\right] \leq \frac{e}{tp} \sM_n(-s,2t,u;G,p)
\nonumber
\end{align}
for every $u \geq 0$, $n\geq 1$, and $0\leq s \leq c_2(p-p_c)t$, where we used the elementary bound $xe^{tx}\leq et^{-1} e^{2tx}$ in the second inequality.
It follows from this and \eqref{eq:diffineq_mgf}  that if $G$ is a connected, locally finite, quasi-transitive graph with $p_c<p_{2\to 2}$ then there exist positive constants $\delta$, $c_1$, $c_2$, and $C_1$ such that
\begin{equation}
\label{eq:diffineq_mgf2}
\partial_p \bE_{p,n}\left[|K_v|e^{u|K_v|}\right] \leq -c_1(p-p_c)\bE_{p,n}\left[|K_v|^2e^{u|K_v|}\right] 
+ \frac{C_1}{t} \sM_n\bigl(-c_2(p-p_c)t,t,u;G,p\bigr)
\end{equation}
for every $v\in V$, $u\geq 0$, $p \in (p_c,p_c+\delta)$, $t\geq 0$ and $1 \leq n < \infty$.

\medskip
In order to apply the inequality \eqref{eq:diffineq_mgf2}, we will need to bound the generating function $\sM_n$. To do this, we derive a family of recursive differential inequalities, \cref{lem:differential_recursion}, which in the next subsection we will use to bound the functions $\sF_{k,n}$ by an inductive argument.

\medskip

 When $n<\infty$ all but finitely many terms of the sum defining $\mathscr{G}_{k,n}(s,t;G,v,p)$ are zero, so that  $\mathscr{G}_{k,n}(s,t;G,v,p)$ is a  differentiable function of $(s,t)$ with $t$-derivative
\begin{align*}
\partial_t \mathscr{G}_{k,n}(s,t;G,v,p) &= 
\sum_{a=0}^\infty \sum_{b=0}^\infty \sum_{\substack{x_1,\ldots,x_k\\ \in V(G)}} \bP_{p,n}^G\Bigl(x_1,\ldots,x_k \in K_v, E_v = a, \Bridges(v,x_1,\ldots,x_k;K_v) =b\Bigr) b e^{sa+tb}.
\\
&=
\bE_{p,n}\left[ \sum_{\substack{x_1,\ldots,x_k\\ \in K_v}} \Bridges(v,x_1,\ldots,x_k;K_v) e^{sE_v+t\Bridges(v,x_1,\ldots,x_k;K_v)}\right].
\end{align*}
The following lemma can be thought of as a sharp form of  \cite[Lemma 2.9]{HermonHutchcroftSupercritical}. 

\begin{lemma}
\label{lem:differential_recursion}
 Let $G$ be a countable graph with degrees bounded by $M$, let $v$ be a vertex of $G$ and let $p \in (0,1)$. Then 
\begin{equation*}
\partial_t \mathscr{G}_{k,n}(s,t;G,v,p) 
\leq \frac{Mp e^t}{1-p} \sum_{\ell=0}^{k-1} \binom{k}{\ell} \mathscr{G}_{\ell+1,n}(s,t;G,v,p) \mathscr{F}_{k-\ell,n}(s,t;G,p)
\end{equation*}
for every $k,n\geq 1$, and $s,t \in \R$.
\end{lemma}

Intuitively, this lemma encodes the fact that we can break up the exploration of a cluster at a bridge edge, first exploring the part of the cluster on the same side of the bridge as the root vertex $v$ then exploring the part of the cluster on the other side of the bridge.
Note that the more complicated form of this inequality will, unfortunately,  make it rather more difficult to analyze than that of \cite[Lemma 2.9]{HermonHutchcroftSupercritical}.



\begin{proof}[Proof of \cref{lem:differential_recursion}]
Fix $k,n \geq 1$, $p \in (0,1)$, $v\in V$, and $s,t \in \R$. For each $a,b \geq 0$ let
\[
R_{k,n}(a,b;G,v,p) = \sum_{\substack{x_1,\ldots,x_k \\\in V(G)}} b \bP_{p}^G\Bigl(x_1,\ldots,x_k \in K_v, E_v = a, \Bridges(v,x_1,\ldots,x_k;K_v) =b\Bigr),
\]
so that
\[
\partial_t \mathscr{G}_{k,n}(s,t;G,v,p)\\ = 
\sum_{a=0}^\infty \sum_{b=0}^\infty e^{sa+tb} R_{k,n}(a,b;G,v,p).
\]

\begin{figure}
\centering
\includegraphics[width=0.5\textwidth]{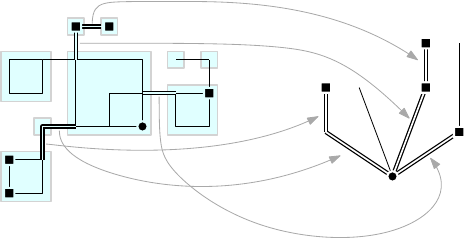}
\caption{Left: A percolation cluster with a root $v$ (black disc) and five marked points $x_1,\ldots,x_5$ (black squares). The two-connected components of the cluster are represented by blue shaded regions. Edges for which the event $\sA_e(x_1,\ldots,x_5)$ holds (for one of the possible orientations of the edge) are drawn as double black lines. Right: The tree of two-connected components of the cluster. Edges contributing to $\Bridges(v,x_1,\ldots,x_5)$ are drawn with double black lines. }
\label{fig:A_e}
\end{figure}

For each oriented edge $e$ of $G$, 
let $K_{e}^-$ and $K_{e}^+$ be the connected components of $e^-$ and $e^+$ in the subgraph of $G$ spanned by the open edges of $G$ other than $e$.
 Thus, $K_{e}^- \neq K_{e}^+$ if and only if $e^-$ and $e^+$ are not connected to each other by an open path not containing $e$.
Let $E_{e}^-$ be the number of edges of $G$ that touch $K_{e}^-$, and let $E_{e}^+$ be the number of edges of $G$ that touch $K_{e}^+$ but do not touch $K_{e}^-$. 
%
For each oriented edge $e$ of $G$ and each $x_1,\ldots,x_k \in V$, let $\sA_e(x_1,\ldots,x_k)$ be the event that $x_1,\ldots,x_k \in K_v$, that $e$ is open, that $v\in K_e^-$, and that there exists $1 \leq i \leq k$ such that $x_i \in K_e^+ \setminus K_e^-$. For each $x_1,\ldots,x_k \in K_v$, the number of oriented edges $e$ such that $\sA_e(x_1,\ldots,x_k)$ holds is precisely $\Bridges(v,x_1,\ldots,x_k;K_v)$ (see \cref{fig:A_e}), so that
we can write
\begin{equation*}
R_{k,n}(a,b;G,v,p) = 
 \sum_{e\in E^\rightarrow } \sum_{\substack{x_1,\ldots,x_k \\ \in V(G)}} \bP_{p,n}^G\Bigl(\sA_e(x_1,\ldots,x_k),  E_v = a, \Bridges(v,x_1,\ldots,x_k;K_v) =b\Bigr).
\end{equation*}
For each strict (possibly empty) subset $A$ of $\{1,\ldots,k\}$, let $\sB_e(x_1,\ldots,x_k;A)$ be the event that $\sA_e(x_1,\ldots,x_k)$ holds and that $x_i \in K_{e}^-$ if and only if $i \in A$ for each $1\leq i \leq k$. Then we can expand
\begin{multline}
R_{k,n}(a,b;G,v,p)=\\ 
\sum_{e\in E^\rightarrow}  \sum_{\substack{x_1,\ldots,x_k \\\in V(G)}} 
\sum_{A \subset \{1,\ldots,k\}} \bP_{p,n}^G\Bigl(\sB_e(x_1,\ldots,x_k;A),  E_v = a, \Bridges(v,x_1,\ldots,x_k;K_v) =b\Bigr).
\label{eq:Brecursion}
\end{multline}
Since the value of this sum does not change when we permute the values of $x_1,\ldots,x_k$, each set $A$ of a given size $\ell$ contributes the same amount to this sum, so that
\begin{multline}
R_{k,n}(a,b;G,v,p)= \\ 
 \sum_{e\in E^\rightarrow} \sum_{\substack{x_1,\ldots,x_k \\\in V(G)}} \sum_{\ell=0}^{k-1}\binom{k}{\ell} \bP_{p,n}^G\Bigl(\sB_e\bigl(x_1,\ldots,x_k;\{1,\ldots,\ell\}\bigr),  E_v = a, \Bridges(v,x_1,\ldots,x_k;K_v) =b\Bigr),
\label{eq:Brecursion2}
\end{multline}
where we interpret $\{1,\ldots,\ell\}$ as the empty set when $\ell=0$.


For each $e \in E^\rightarrow$, $0 \leq \ell \leq k-1$, each $y_1,\ldots,y_\ell \in V(G)$, each $z_1,\ldots,z_{k-\ell} \in V(G)$, and each $a_1,a_2,b_1,b_2 \geq 0$, consider the events 
\begin{multline*}\sC_{e,\ell}(y_1,\ldots,y_\ell;a_1,b_1)
\\:= \Bigl\{\text{$v\in K_e^-$, $y_i \in K_{e}^-$  for every $1 \leq i \leq \ell$, $E_e^-=a_1$, and  $\Bridges(v,y_1,\ldots,y_\ell,e^-;K_{e}^-)=b_1$}\Bigr\}
\end{multline*} and 
\begin{multline*}\sD_{e,\ell}(z_1,\ldots,z_{k-\ell};a_2,b_2)
\\:= \Bigl\{\text{$z_i \in K_{e}^+ \setminus K_e^-$  for every $1 \leq i \leq k-\ell$, $E_e^+=a_2$, and  $\Bridges(e^+,z_1,\ldots,z_{k-\ell};K_{e}^+)=b_2$}\Bigr\},
\end{multline*}
where we recall that $E_{e}^-$ is the number of edges of $G$ that touch $K_{e}^-$ and $E_{e}^+$ is the number of edges of $G$ that touch $K_{e}^+$ but do not touch $K_{e}^-$. 
Observe that the event $\sB_e\bigl(x_1,\ldots,x_k;\{1,\ldots,\ell\}\bigr) \cap \{ E_v = a, |\Bridges(v,x_1,\ldots,x_k;K_v)| =b\}$ can be rewritten as the disjoint union
\begin{multline}
\sB_e\bigl(x_1,\ldots,x_k;\{1,\ldots,\ell\}\bigr) \cap \Bigl\{  E_v = a, |\Bridges(v,x_1,\ldots,x_k;K_v)| =b\Bigr\} \\= \bigcup_{a_1=0}^a \bigcup_{b_1=0}^{b-1} \left[\{e \text{ open}\} \cap \sC_{e,\ell}\Bigl(x_1,\ldots,x_\ell;a_1,b_1\Bigr) \cap \sD_{e,\ell}\Bigl(x_{\ell+1},\ldots,x_k;a-a_1,b-b_1-1\Bigr) \right].
\label{eq:B_as_CD_events}
\end{multline}
Indeed, this follows from the observation that if 
$\sB_e\bigl(x_1,\ldots,x_k;\{1,\ldots,\ell\}\bigr)$ holds then $E_v = E_e^- + E_e^+$ and $\Bridges(v,x_1,\ldots,x_k;K_v) = \Bridges(v,x_1,\ldots,x_\ell,e^-;K_e^-)+ \Bridges(e^+,x_{\ell+1},\ldots,x_k;K_e^+)+1$.
Noting that the random variable $\omega(e)$ is independent of the pair of random variables $(K_e^-,K_e^+)$, we deduce from \eqref{eq:Brecursion2} and \eqref{eq:B_as_CD_events} that
\begin{multline*}
R_{k,n}(a,b;G,v,p) = p\sum_{\ell=0}^{k-1} \binom {k}{\ell} \sum_{e\in E^\rightarrow} \sum_{a_1=0}^a \sum_{b_1=0}^{b-1}\sum_{\substack{y_1,\ldots,y_\ell \\\in V(G)}} \sum_{\substack{z_1,\ldots,z_{k-\ell} \\\in V(G)}} \\ \bP_{p,n}^G\left(\sC_{e,\ell}\bigl(y_1,\ldots,y_\ell;a_1,b_1\bigr) \cap \sD_{e,\ell}\bigl(z_1,\ldots,z_{k-\ell};a-a_1,b-b_1-1\bigr)\right)
\end{multline*}
and hence that
\begin{multline}
\partial_t \sG_{k,n}(s,t;G,v,p) = pe^t\sum_{\ell=0}^{k-1} \binom {k}{\ell} \sum_{e\in E^\rightarrow} \sum_{a_1=0}^\infty \sum_{b_1=0}^\infty\sum_{\substack{y_1,\ldots,y_\ell \\\in V(G)}} \sum_{a_2=0}^\infty \sum_{b_2=0}^\infty\sum_{\substack{z_1,\ldots,z_{k-\ell} \\\in V(G)}}e^{sa_1+tb_1} e^{sa_2+tb_2}\\ \bP_{p,n}^G\left(\sC_{e,\ell}\bigl(y_1,\ldots,y_\ell;a_1,b_1\bigr) \cap \sD_{e,\ell}\bigl(z_1,\ldots,z_{k-\ell};a_2,b_2\bigr)\right) .
\end{multline}

Let $\cF_e^-$ be the $\sigma$-algebra generated by the random variable $K_e^-$ and let $H_e^+$ be the random subgraph of $G$ spanned by those edges of $G$ that do not touch $K_e^-$. The conditional distribution of $K_e^+ \setminus K_e^-$ given $\cF_e^-$ coincides with that of the cluster of $e^+$ in Bernoulli-$p$ bond percolation on $H_e^+$, so that
\begin{align*}
\sum_{a_2=0}^\infty &\sum_{b_2=0}^\infty \sum_{\substack{z_1,\ldots,z_{k-\ell} \\\in V(G)}} \bP_{p,n}^G\left(\sD_{e,\ell}\bigl(z_1,\ldots,z_{k-\ell};a_2,b_2\bigr) \mid \cF_e^-\right) e^{sa_2+tb_2}
\\
&= \sum_{\substack{z_1,\ldots,z_{k-\ell} \\\in V(H_e^+)}} \bP_{p,n-E_v^-}^{H_e^+}\left(z_1,\ldots,z_{k-\ell} \in K_{e^+}, E_{e^+} = a_2, \Bridges(e^+,z_1,\ldots,z_{k-\ell};K_{e^+})=b_2\right) e^{sa_2+tb_2}
\\&= \sG_{k-\ell,n-E_v^-}(s,t;H^+_e,e^+,p) \leq \sF_{k-\ell,n-E_v^-}(s,t;G,p) \leq \sF_{k-\ell,n}(s,t;G,p)
\end{align*}
almost surely. Since $\sC_{e,\ell}\bigl(y_1,\ldots,y_\ell;a_1,b_1\bigr)$ is $\cF_e^-$-measurable, it follows that
\begin{multline}
\partial_t \sG_{k,n}(s,t;G,v,p) \leq pe^t\sum_{\ell=0}^{k-1} \binom {k}{\ell}
\sF_{k-\ell,n}(s,t;G,p)
  \\\sum_{e\in E^\rightarrow} \sum_{a_1=0}^\infty \sum_{b_1=0}^\infty
\sum_{\substack{y_1,\ldots,y_\ell \\\in V(G)}} \bP_{p,n}^G\left(\sC_{e,\ell}\bigl(y_1,\ldots,y_\ell;a_1,b_1\bigr)\right) e^{sa_1+tb_1}.
\label{eq:DtoF}
\end{multline}
On the other hand, since $\omega(e)$ is independent of $\cF_e^-$ we have that
\begin{align*}\bP_{p,n}^G\Bigl(\sC_{e,\ell}\bigl(&y_1,\ldots,y_\ell;a_1,b_1\bigr)\Bigr)\\ &= \frac{1}{1-p}\bP_{p,n}^G\left(\sC_{e,\ell}\bigl(y_1,\ldots,y_\ell;a_1,b_1\bigr) \cap \{e \text{ closed}\}\right) 
\\
&=\frac{1}{1-p}\bP_{p,n}^G\Bigl(\text{$e$ closed, $y_1,\ldots,y_\ell,e^- \in K_v$, $E_v=a_1$, and  $\Bridges(v,y_1,\ldots,y_\ell,e^-;K_{v})=b_1$}\Bigr)\\
&\leq \frac{1}{1-p}\bP_{p,n}^G\Bigl(\text{$y_1,\ldots,y_\ell,e^- \in K_v$, $E_v=a_1$, and  $\Bridges(v,y_1,\ldots,y_\ell,e^-;K_{v})=b_1$}\Bigr) 
\end{align*}
and hence that
\begin{align}
&\sum_{e\in E^\rightarrow} \sum_{a_1=0}^\infty \sum_{b_1=0}^\infty
\sum_{\substack{y_1,\ldots,y_\ell \\\in V(G)}}  \bP_{p,n}^G\left(\sC_{e,\ell}\bigl(y_1,\ldots,y_\ell;a_1,b_1\bigr)\right) e^{sa_1+tb_1}
\nonumber 
\\
&\leq \frac{M}{1-p} \sum_{a_1=0}^\infty \sum_{b_1=0}^\infty
\sum_{\substack{y_1,\ldots,y_{\ell+1}\\ \in V(G)}}  \bP_{p,n}^G\left(y_1,\ldots,y_\ell,y_{\ell+1} \in K_v, E_v = a_1, \Bridges(v,y_1,\ldots,y_{\ell+1};K_v) =b_1\right) e^{sa_1+tb_1}
\nonumber
\\
&= \frac{M}{1-p}\sG_{\ell+1,n}(s,t;G,v,p).
\label{eq:CtoG}
\end{align}
Substituting \eqref{eq:CtoG} into \eqref{eq:DtoF} completes the proof.
\end{proof}

We will use \cref{lem:differential_recursion} in the following integral form.

\begin{corollary}
\label{cor:integral_recursion}
 Let $G$ be a countable graph with degrees bounded by $M$ and let $p \in (0,1)$. Then 
\begin{multline}
\label{eq:recursion2}
\sF_{k,n}(s,t_2;G,p) - \sF_{k,n}(s,t_1;G,p) \\\leq \frac{Mp}{1-p} \sum_{\ell=0}^{k-1} \binom{k}{\ell}\int_{t_1}^{t_2} e^t  \mathscr{F}_{\ell+1,n}(s,t;G,p) \mathscr{F}_{k-\ell,n}(s,t;G,p) \dif t
\end{multline}
for every $k,n\geq 1$ and $s,t_1,t_2 \in \R$ with $t_1 \leq t_2$.
\end{corollary}

\begin{proof}[Proof of \cref{cor:integral_recursion}] We have trivially that
\begin{align*}
\sF_{k,n}(s,t_2;G,p) - \sF_{k,n}(s,t_1;G,p) &= \sup_{H,v} \sG_{k,n}(s,t_2;H,v,p) - \sup_{H,v}\sG_{k,n}(s,t_1;H,v,p)\\
&\leq \sup_{H,v} \left[\sG_{k,n}(s,t_2;H,v,p) - \sG_{k,n}(s,t_1;H,v,p)\right]
\end{align*}
and applying \cref{lem:differential_recursion} yields that
\begin{align*}
\sF_{k,n}(s,t_2;G,p) - \sF_{k,n}(s,t_1;G,p)&\leq \sup_{H,v} \int_{t_1}^{t_2} \frac{Mp e^t}{1-p} \sum_{\ell=0}^{k-1} \binom{k}{\ell} \mathscr{G}_{\ell+1,n}(s,t;H,v,p) \mathscr{F}_{k-\ell,n}(s,t;H,p) \dif t.
\end{align*}
The claim follows since $\mathscr{G}_{\ell+1,n}(s,t;H,v,p) \mathscr{F}_{k-\ell,n}(s,t;H,p) \leq\mathscr{F}_{\ell+1,n}(s,t;G,p) \mathscr{F}_{k-\ell,n}(s,t;G,p)$ for every subgraph $H$ of $G$.
\end{proof}

\subsection{Bounding the positive term II: analysis of the auxiliary and main differential inequalities}
\label{subsec:positivetermII}

In this subsection we complete the proof of \cref{prop:vol_supercritical_upper}. The main step will be to prove the following proposition via an analysis of the recursive differential inequality provided by \cref{lem:differential_recursion}.
This proposition (or, more accurately, the intermediate inequality \eqref{eq:STP_F}) serves as a sharp quantitative version of \cite[Equation 2.21]{HermonHutchcroftSupercritical} under the additional assumption that $\nabla_{p_c}<\infty$. The generating function $\sM_n$ was defined in \eqref{eq:Mdef}.

\begin{prop} 
\label{prop:generating_function_estimate}
Let $G$ be an infinite, connected, locally finite, quasi-transitive graph such that $\nabla_{p_c} < \infty$, and let $\alpha \geq 0$. Then there exist positive constants $c_1=c_1(G,\alpha)$, $c_2=c_2(G,\alpha)$, $C=C(G,\alpha)$, and $\delta=\delta(G,\alpha)$ such that
\[
\sM_n(-c_1\eps^2,\alpha c_1\eps,c_2\eps^2;G,p_c+\eps) \leq C \eps^{-1}
\]
for every $n \geq 1$ and $0<\eps \leq \delta$.
\end{prop}



Before proceeding further, let us see how \cref{prop:generating_function_estimate} can be used to complete the proof of \cref{prop:vol_supercritical_upper}.

\begin{proof}[Proof of \cref{prop:vol_supercritical_upper} given \cref{prop:generating_function_estimate}]
Fix $v\in V$. 
By \eqref{eq:diffineq_mgf2}  there exist positive constants $c_1$, $c_2$, $\delta_1$, and $C_1$ such that
\begin{align*}
\partial_p \bE_{p,n}\left[|K_v|e^{u|K_v|}\right] 
&\leq -c_1(p-p_c)\bE_{p,n}\left[|K_v|^2e^{u|K_v|}\right] 
+ \frac{C_1}{t} \sM_n\bigl(-c_2(p-p_c)t,t,u;G,p\bigr)
\\
&= -c_1(p-p_c)\partial_u \bE_{p,n}\left[|K_v|e^{u|K_v|}\right] 
+ \frac{C_1}{t} \sM_n\bigl(-c_2(p-p_c)t,t,u;G,p\bigr)
\end{align*}
for every $u\geq 0$, $p \in (p_c,p_c+\delta_1)$, $t\geq 0$ and $1 \leq n < \infty$. On the other hand, applying \cref{prop:generating_function_estimate} with $\alpha = c_2^{-1}$ yields that there exist positive constants $c_3$, $c_4$, $\delta_2$, and $C_2$ such that
\begin{align*}
\sM_n\bigl(-c_2c_3(p-p_c)^2,c_3(p-p_c),u;G,p\bigr)  &\leq \sM_n\bigl(-c_2c_3(p-p_c)^2,c_3(p-p_c),c_4(p-p_c)^2;G,p\bigr)\\ &\leq C_2 (p-p_c)^{-1}
\end{align*}
for every $p \in (p_c,p_c+\delta_2)$ and $0 \leq u \leq c_4(p-p_c)^2$. It follows that there exists a constant $C_3$ such that
\begin{align}
\partial_p \bE_{p,n}\left[|K_v|e^{u|K_v|}\right] 
&\leq-c_1(p-p_c) \partial_u \bE_{p,n}\left[|K_v|e^{u|K_v|}\right] 
+ C_3 (p-p_c)^{-2}
\label{eq:pde_endgame}
\end{align}
for every $p\in (p_c,p_c+\delta_1 \wedge \delta_2)$ and $0 \leq u \leq c_4(p-p_c)^2$. 

\medskip

By \cref{cor:good_points}, there exists $\delta_3>0$ and $C_4 < \infty$ such that for every $0<\eps \leq \delta_3$ there exists $p_0=p_0(\eps) \in (p_c+\eps/4,p_c+\eps/2)$ such that $\bE_{p_0,\infty}|K_v| \leq C_4 \eps^{-1}$. Let $\delta = \min\{\delta_1,\delta_2,\delta_3\}$
and $c_5=(c_1 \wedge c_4)/2$. Let $0<\eps \leq \delta$ and let $p_0=p_0(\eps)$. It follows by the chain rule that
\[\frac{d}{dp} \bE_{p,n}\left[|K_v|\exp\left[c_5(p-p_0)^2|K_v|\right]\right] \leq C_3(p-p_c)^{-2} \]
for every $n\geq 1$ and $p_0 \leq p \leq p_c+\eps$.
Integrating this differential inequality between $p_0$ and $p_c+\eps$ and noting that $\eps/2 \leq p_c+\eps -p_0 \leq \eps$ yields that
\begin{align*}
\bE_{p_c+\eps,n}\left[|K_v|\exp\left[\frac{1}{4}c_5\eps^2|K_v|\right]\right]&\leq 
\bE_{p_c+\eps,n}\left[|K_v|\exp\left[c_5(p_c+\eps-p_0)^2|K_v|\right]\right] 
\\
&=
\bE_{p_0}\left[|K_v|\right] + \int_{p_0}^{p_c+\eps} \frac{d}{dp} \bE_{p,n}\left[|K_v|\exp\left[c_5(p-p_0)^2|K_v|\right]\right] \dif p 
\\
&\leq 
 C_4\eps^{-1} + \int_{p_0}^{p_c+\eps} C_3 (p-p_c)^{-2} \dif p
 \leq (C_4+16C_3)\eps^{-1}
\end{align*}
for every $n\geq 1$ and $0< \eps \leq \delta$ as required. \qedhere

\end{proof}

We now begin to work towards the proof of \cref{prop:generating_function_estimate}, which will rely on an inductive analysis of the integral inequality of \cref{cor:integral_recursion}. This analysis will require the following two lemmas as input: The first applies \cref{lem:inside_window} to analyze $\sF_{k,\infty}$ when $t=0$ and $s< 0$, and the second applies \cref{lem:SkinnyRadius} to establish the $k=1$ base case.

\begin{lemma}[Boundary conditions]
\label{lem:t_zero}
Let $G$ be an infinite, connected, locally finite quasi-transitive graph such that $\nabla_{p_c} < \infty$. Then there exist positive constants $C$, and $\delta$ such that
\begin{equation}
\sF_{k,\infty}\bigl(-\lambda \eps^2,0;G,p_c+\eps\bigr) \leq k! C^k \lambda^{-k} \eps^{-2k+1}
\end{equation}
for every $k \geq 1$, $0<\eps \leq \delta$, and $0< \lambda \leq 1$.
\end{lemma}

\begin{lemma}[Base case]
\label{lem:base_case}
Let $G$ be an infinite, connected, locally finite, quasi-transitive graph such that $\nabla_{p_c} < \infty$. Then there exist positive constants $c$, $C$, and $\delta$ such that 
\begin{equation}
\label{eq:generating_base}
\sF_{1,\infty}\bigl(-\lambda \eps^2, \alpha \lambda \eps;G,p_c+\eps\bigr) \leq C \lambda^{-1} \eps^{-1}
\end{equation}
for every $\alpha \geq 0$, $0<\eps \leq \delta$, and $0< \lambda \leq 1 \wedge c \alpha^{-1} e^{-C \alpha}$.
\end{lemma}

The proofs of both lemmas will use the fact that if $X$ is a non-negative random variable then
\begin{equation}
\label{eq:integration_by_parts}
\E\left[X^k e^{sX}\right] = \int_{0}^\infty (k+st)t^{k-1} e^{st} \P(X \geq t) \dif t
\end{equation}
for every $k\geq 1$ and $s \in \R$, where it is possible that both sides are equal to $+\infty$ when $s \geq 0$. This identity is a standard consequence of the integration-by-parts formula.

\begin{proof}[Proof of \cref{lem:t_zero}]
Let $p_c=p_c(G)$. We have by \cref{lem:inside_window} that there exist positive constants $C_1$ and $\delta$ such that
\[
\bP_{p_c+\eps}^H(1+E_v \geq u) \leq C_1\left[u^{-1/2} + \eps \right]
\]
for every subgraph $H$ of $G$, every vertex $v$ of $H$, every $u \geq 0$ and every $0<\eps \leq \delta$. Since $|K_v| \leq 1 + E_v$, we deduce by standard calculations that
\begin{align*}
\sG_{k,\infty}(-s,0;H,v,p_c+\eps) &= 
\sum_{a=0}^\infty \sum_{\substack{x_1,\ldots,x_k \\\in V(G)}} \bP_{p}^G\Bigl(x_1,\ldots,x_k \in K_v, E_v = a \Bigr) e^{-sa} = \bE_{p_c+\eps} \left[|K_v|^k e^{-sE_v}\right]
\\
&\leq
e^s\bE_{p_c+\eps}^H\left[(1+E_v)^k e^{-s (1+E_v)}\right] 
\\&= e^s \int_0^\infty (k-su) u^{k-1} e^{-su} \bP_p^H(1+E_v \geq u)  \dif  u\\
&
\leq  C_1 e^s  \left[\int_0^{\infty} k u^{k-3/2}  e^{-su} \dif u + \eps \int_{0}^\infty k u^{k-1}e^{-su}\dif u   \right]
\end{align*}
for every $s > 0$.
Using the identities  $\int_0^\infty u^{a-1} e^{-su} \dif u = s^{-a} \int_0^\infty y^{a-1} e^{-y} \dif y = s^{-a}\Gamma(a)$ and $\Gamma(k)=(k-1)!$ we obtain that
\begin{equation*}
\sG_{k,\infty}(-s,0;H,v,p_c+\eps) \leq C_1 e^s \left[k s^{-k+1/2} \Gamma(k-1/2)
+
k \eps s^{-k} \Gamma(k)  
  \right]\\ \leq  C_1 e^s k! \left[s^{-k+1/2}+ \eps s^{-k}  \right]
\end{equation*}
for every subgraph $H$ of $G$, every vertex $v$ of $H$, every $n \geq 0$, every $0<\eps \leq \delta$, and every $s \geq 0$. The claim follows by taking $s=\lambda \eps^2$.
\end{proof}


\begin{proof}[Proof of \cref{lem:base_case}]
Fix $\alpha \geq 0$, a subgraph $H$ of $G$ and a vertex $v$ of $H$. Letting $R_v$ denote the intrinsic radius of the cluster of $v$ in $H$ (i.e., the maximal intrinsic distance of a vertex of $K_v$ from $v$), we trivially have that $\Bridges(v,x;K_v) \leq R_v$ for every $x\in K_v$, so that
\begin{align*}
\sG_{1,\infty}(-\lambda \eps^2,\alpha \lambda \eps ; H,v,p_c+\eps) 
=
 \bE_{p_c+\eps}^H\left[e^{-\lambda \eps^2 E_v} \sum_{x \in K_v} e^{\alpha\lambda \eps\Bridges(v,x;K_v)}\right]
 \leq
  \bE_{p_c+\eps}^H\left[|K_v| e^{-\lambda \eps^2 E_v} e^{\alpha\lambda \eps R_v}\right]
\end{align*}
and hence that
\begin{align}
&\sG_{1,\infty}(-\lambda\eps^2,\alpha\lambda \eps ; H,v,p_c+\eps) 
\nonumber
\\
&\hspace{0.8cm}\leq \bE_{p_c+\eps}^H\left[|K_v| e^{-\lambda \eps^2 E_v} e^{\alpha\lambda \eps R_v}\mathbbm{1}\left(R_v \leq \frac{\eps E_v}{2 \alpha} \right)\right]+ \bE_{p_c+\eps}^H\left[|K_v| e^{-\lambda \eps^2 E_v} e^{\alpha\lambda \eps R_v}\mathbbm{1}\left(R_v > \frac{\eps E_v}{2 \alpha} \right)\right]
\nonumber
\\
 &\hspace{0.8cm}\leq
 \bE_{p_c+\eps}^H\left[|K_v| e^{- \lambda \eps^2 E_v /2} \right]
 +
  \bE_{p_c+\eps}^H\left[|K_v|  e^{\alpha\lambda \eps R_v}\mathbbm{1}\left(R_v > \frac{\eps E_v}{2 \alpha} \right)\right]
  \label{eq:base1}
\end{align}
for every $0<\eps \leq \delta_1$ and $0<\lambda \leq 1$.
For the first term, \cref{lem:t_zero} implies that there exist positive  constants $\delta_1$ and $C_1$ such that 
\begin{equation}
\bE_{p_c+\eps}^H\left[|K_v| e^{- \lambda \eps^2 E_v /2} \right] \leq \sF_{1,\infty}\left(-\lambda\eps^2/2,0;G,p_c+\eps\right) \leq C_1 \lambda^{-1} \eps^{-1}
\label{eq:base2}
\end{equation}
for every $0 < \eps \leq \delta_1$ and $0 < \lambda \leq 1$.

 For the second term, we first decompose further
\begin{align*}
  \bE_{p_c+\eps}^H\left[|K_v|  e^{\alpha\lambda \eps R_v}\mathbbm{1}\left(R_v > \frac{\eps E_v}{2 \alpha} \right)\right] 
  &=  \bE_{p_c+\eps}^H\left[|K_v|  e^{\alpha\lambda \eps R_v}\mathbbm{1}\left(R_v > \frac{\eps E_v}{2 \alpha}, R_v < \eps^{-1} \right)\right] 
  \\&\hspace{1.6cm}+
   \bE_{p_c+\eps}^H\left[|K_v|  e^{\alpha\lambda \eps R_v}\mathbbm{1}\left(R_v > \frac{\eps E_v}{2 \alpha}, R_v \geq \eps^{-1} \right)\right] 
  \\
   &=: \mathrm{I} + \mathrm{II},
\end{align*}
where the second inequality means that we write $\mathrm{I}$ and $\mathrm{II}$ for the first and second terms appearing on the right hand side of the first equality.
To bound the term $\mathrm{I}$, we apply \eqref{eq:ballwindow} to deduce that there exist positive constants $\delta_2$ and $C_2$ such that
\begin{equation}
\mathrm{I} \leq e^{\alpha \lambda} \bE_{p_c+\eps}^H \left[|K_v| \mathbbm{1}(R_v < \eps^{-1}) \right]
\leq  e^{\alpha \lambda} \bE_{p_c+\eps}^H \left[ \# B_\mathrm{int}(v,\lfloor \eps^{-1} \rfloor) \right] \leq C_2 \eps^{-1} e^{\alpha \lambda},
\label{eq:I}
\end{equation}
for every $0 \leq \eps \leq \delta_2$ and $0<\lambda \leq 1$. Finally, to bound the term $\mathrm{II}$, we note that
\[
\mathrm{II} \leq \frac{4\alpha}{\eps}\bE_{p_c+\eps}^H\left[R_v  e^{\alpha\lambda \eps R_v}\mathbbm{1}\left(R_v > \frac{\eps E_v}{2 \alpha}, R_v \geq \eps^{-1} \right)\right]
\]
(where we used the fact that $|K_v| \leq E_v+1 \leq 2E_v$ when $R_v \geq \eps^{-1}>0$)
and hence by \eqref{eq:integration_by_parts} that there exists a constant $C_3$ such that
\[
\mathrm{II} \leq C_3 \frac{\alpha}{\eps} \int_{\eps^{-1}}^\infty (1+\alpha\lambda \eps t) e^{\alpha\lambda\eps t} \bP_{p_c+\eps}^H\left(t \leq R_v < \infty,\, E_v < \frac{2\alpha}{\eps} R_v \right) \dif t.
\]
We then apply \cref{lem:SkinnyRadius}  to obtain that there exist positive constants $c_1$, $C_4$, $C_5$ and $\delta_3$ such that
\[
\mathrm{II} 
\leq C_5 \frac{\alpha}{\eps}\int_{\eps^{-1}}^\infty (\eps+\alpha\lambda \eps^2 t) \exp\left[-2c_1 e^{-2 C_4 \alpha} \eps t+\alpha\lambda \eps t\right] \dif t 
\]
for every $0<\eps \leq \delta_3$ and $0<\lambda \leq 1$.
We deduce in particular that if $0<\eps \leq \delta_3$ and $\alpha \lambda \leq c_1 e^{-2C_4 \alpha}$ then
\begin{align*}
\mathrm{II} 
&\leq C_5  \frac{\alpha}{\eps}\int_{\eps^{-1}}^\infty \bigl(\eps + \alpha \lambda \eps^2 t) \exp\left[-c_1 e^{-2 C_4 \alpha} \eps t\right] \dif t.
\end{align*}
Similarly to the proof of \cref{lem:t_zero}, using the identities
$\int_0^\infty e^{-st} \dif t = s^{-1}$ and $\int_0^\infty t e^{-st} \dif t = s^{-2}$ yields that there exists a constant $C_6$ such that
\begin{align}
\mathrm{II}
&\leq C_6 
 \frac{\alpha}{\eps}\left[
e^{2C_4 \alpha} + \alpha\lambda e^{4C_4 \alpha} \right] \leq C_6\alpha(1+\alpha \lambda) e^{4C_4 \alpha} \eps^{-1}
\label{eq:II}
\end{align}
%
%
for every $0<\eps \leq \delta_3$ and $0<\lambda \leq 1 \wedge c_1 \alpha^{-1} e^{-2C_4 \alpha}$. 

Putting together all the estimates \eqref{eq:base1}, \eqref{eq:base2}, \eqref{eq:I}, and \eqref{eq:II}, we obtain that there exist positive constants $\delta_4=\delta_1 \wedge \delta_2 \wedge \delta_3$ and $C_7$ such that
\[
\sG_{1,\infty}(-\lambda\eps^2,\alpha\lambda\eps;H,v,p_c+\eps) \leq C_7 \left[\lambda^{-1}  + e^{\alpha \lambda}  + \alpha(1+\alpha \lambda) e^{4C_4 \alpha} \right] \eps^{-1}
\]
for every $0<\eps \leq \delta_4$ and $0<\lambda \leq 1 \vee c_1 \alpha^{-1} e^{-2C_4 \alpha}$. 
It follows that there exists a constant $C_8$ such that if $0< \eps \leq \delta_4$ and $0< \lambda \leq 1 \wedge c_1 \alpha^{-1} e^{-4C_4 \alpha} \wedge e^{-\alpha} $ then
\[
\sG_{1,\infty}(-\lambda\eps^2,\alpha\lambda\eps;H,v,p_c+\eps) \leq C_8 \lambda^{-1} \eps^{-1}.
\]
Since $H$, $v$, and $\alpha\geq 0$ were arbitrary, this is easily seen to imply the claim.
\end{proof}

We are now ready to prove \cref{prop:generating_function_estimate}.

\begin{proof}[Proof of \cref{prop:generating_function_estimate}]
Let $G$ be an infinite, connected, locally finite, quasi-transitive graph,  let $p_c=p_c(G)$, and suppose that $\nabla_{p_c} < \infty$. Let $\alpha \geq 0$. It suffices to prove there exist positive constants $c=c(G,\alpha)$, $C=C(G,\alpha)$, and $\delta=\delta(G,\alpha)$ such that
\begin{equation}
\label{eq:STP_F}
\sF_{k,\infty}\left(-c\eps^2,\alpha c\eps;G,p_c+\eps\right) \leq (k-1)! C^k \eps^{-2k+1}
\end{equation}
for every $k\geq 1$ and $0<\eps \leq \delta$. 
Indeed, the claim will then follow by noting that if \eqref{eq:STP_F} holds then
\[
\sM_\infty\bigl(-c\eps^2,\alpha c\eps,\frac{1}{2C}\eps^2;G,p_c(G)+\eps\bigr)
\leq \sum_{k=0}^\infty \frac{\eps^{2k}}{k! 2^kC^k} k! C^{k+1} \eps^{-2(k+1)+1} = 2C \eps^{-1}
\]
for every $0<\eps \leq \delta$.
The case $\alpha =0$ is handled by \cref{lem:t_zero}, so we may suppose that $\alpha>0$.


 All the constants appearing in this proof will depend only on $\alpha$ and $G$.   Throughout the proof we will also use the convention that $(-1)!=1$. 
We begin by noting that, with this convention, there exists a finite constant $C_1$ such that
\begin{multline*}
\frac{1}{(k-1)!}\sum_{\ell=1}^{k-1} \binom{k}{\ell} (\ell-1)!(k-\ell-2)! = \sum_{\ell=1}^{k-2} \frac{k}{\ell (k-\ell)(k-\ell-1)} + \frac{k}{k-1}\\
\leq \frac{4}{k-1} \sum_{\ell=1}^{\lfloor (k-1)/2  \rfloor}  \frac{1}{\ell} + \sum_{\ell=\lceil (k-1)/2\rceil}^{k-2} \frac{2k}{(k-1)(k-\ell)(k-\ell-1)} + \frac{k}{k-1}\leq C_1
\end{multline*}
for every $k \geq 2$. By \cref{lem:base_case}, there exist positive constants $c_1$, $C_2$, and $\delta_1$ such that
\begin{equation}
\sF_{1,\infty}(-\lambda\eps^2,\alpha\lambda\eps;G,p_c+\eps) \leq C_2 \eps^{-1}
\label{eq:base_case_restatement}
\end{equation}
for every $0\leq \lambda \leq c_1$ and $0<\eps \leq \delta_1$.
Define 
\[
c = \min \left\{ 1, c_1, \frac{1-p_c}{4 \alpha C_2 eM}, \frac{1-p_c}{4 \alpha C_1 eM} \right\},
\]
where $M$ is the maximum degree of $G$. 
For each $k,n \geq 1$ and $0<\eps \leq \delta_1$ we define increasing functions $f_{k,n,\eps}:[0,1]\to[0,\infty)$ and $f_{k,\eps}:[0,1]\to [0,\infty]$ by 
\begin{align*}f_{k,n,\eps}(\theta)=\sF_{k,n}(-c\eps^2,\alpha c\eps\theta;G,p_c(G)+\eps)
\qquad \text{ and } \qquad
f_{k,\eps}(\theta) = \sF_{k,\infty}(-c\eps^2,\alpha c\eps\theta;G,p_c(G)+\eps),\end{align*}
so that $f_{k,\eps}(\theta)= \sup_{n\geq 1} f_{k,n,\eps}(\theta)$. 
Thus,  \eqref{eq:base_case_restatement} implies that $f_{1,\eps}(\theta) \leq f_{1,\eps}(1) \leq C_2 \eps^{-1}$ for every $\eps \leq \delta_1$ and $\theta \in [0,1]$.
Meanwhile, \cref{lem:t_zero} easily implies that there exist positive constants $C_3$ and $\delta_2$, such that
\begin{equation}
f_{k,\eps}(0) \leq C_3^{k} \eps^{-2k+1}(k-2)!
\label{eq:thetazero}
\end{equation}
for every $0<\eps \leq \delta_2$ and $k\geq 1$.


Let 
$\delta = \min\{\delta_1,\delta_2,(1-p_c)/2\}$ and let $C_4 = 4(C_2 \vee C_3)> 0$.
 It suffices to prove that
\begin{equation}
\label{eq:generating_induction_hypothesis}
f_{k,\eps}(\theta) \leq C_4^{k} e^{ C_4(k-1)\theta} \eps^{-2k+1} (k-2)!
\end{equation}
for every $k,n \geq 1$, $\theta \in [0,1]$ and $0<\eps \leq \delta$. We will do this by induction on $k$. 
The base case $k=1$ follows immediately from \eqref{eq:base_case_restatement}. Suppose that $k\geq 2$ and that the claim holds for all $1\leq k' <k$. Fix $n\geq 1$ and $0<\eps \leq \delta$. It follows from \cref{cor:integral_recursion} that
\begin{align*}
f_{k,n,\eps}(\theta) \leq f_{k,n,\eps}(0)+\frac{M(p_c+\eps)}{1-(p_c+\eps)} \sum_{\ell=0}^{k-1} \binom{k}{\ell}  \int_{0}^\theta e^{c\eps \varphi} f_{\ell+1,n,\eps}(\varphi)f_{k-\ell,n,\eps}(\varphi) \alpha c \eps\dif \varphi,
\end{align*}
where the $\alpha c \eps$ term comes from changing variables in the integral from $t$ to $\varphi$. Our choice of $c$ therefore yields that
\begin{align}
&\hspace{-0.5em}f_{k,n,\eps}(\theta)
\nonumber
\\
&\leq 
f_{k,\eps}(0) + \frac{eM \alpha c\eps}{1-p_c-\delta} \int_{0}^\theta f_{k,n,\eps}(\varphi) f_{1,\eps}(\varphi) \dif \varphi
+\frac{eM \alpha c \eps}{1-p_c-\delta} \sum_{\ell=1}^{k-1} \binom{k}{\ell}\int_{0}^\theta  f_{\ell+1,\eps}(\varphi)f_{k-\ell,\eps}(\varphi) \dif \varphi
\nonumber
\\
&\leq 
f_{k,\eps}(0)+  \frac{\eps}{4C_2} \int_{0}^\theta f_{k,n,\eps}(\varphi) f_{1,\eps}(\varphi) \dif \varphi + \frac{\eps}{4C_1} \sum_{\ell=1}^{k-1} \binom{k}{\ell}\int_{0}^\theta  f_{\ell+1,\eps}(\varphi)f_{k-\ell,\eps}(\varphi) \dif \varphi
\label{eq:induction1}
\end{align}
for every $\theta \in [0,1]$, where we used that $e^{c\eps \varphi} \leq e^c \leq e$ in the first line. 
The first two terms are easily bounded by using \eqref{eq:base_case_restatement}, \eqref{eq:thetazero}, and the definition of $C_4$ to obtain that
\begin{align}
f_{k,\eps}(0)+  \frac{\eps}{4C_2} \int_{0}^\theta f_{k,n,\eps}(\varphi) f_{1,\eps}(\varphi) \dif \varphi 
&\leq C_3^k\eps^{-2k+1}(k-2)! + \frac{\eps}{2C_2}\int_{0}^\theta f_{k,n,\eps}(\theta) C_2 \eps^{-1} \dif \varphi
\nonumber
\\
&\leq \frac{1}{4} C_4^{k} \eps^{-2k+1} (k-2)! + \frac{1}{2}f_{k,n,\eps}(\theta)
\label{eq:induction2}
\end{align}
for every $\theta \in [0,1]$.
For the final term, we apply the induction hypothesis and the definition of $C_1$ 
to obtain that
\begin{align}
\frac{\eps}{4C_1} \sum_{\ell=1}^{k-1} &\binom{k}{\ell} \int_{0}^\theta  f_{\ell+1,\eps}(\varphi)f_{k-\ell,\eps}(\varphi) \dif \varphi 
\nonumber
\\
&\hspace{-0.25em}\leq
\frac{\eps}{4C_1} \sum_{\ell=1}^{k-1} \binom{k}{\ell}\int_{0}^\theta  C_4^{\ell+1}e^{C_4\ell\varphi}\eps^{-2\ell-1} (\ell-1)!C_4^{k-\ell}e^{C_4(k-\ell-1)\varphi}\eps^{-2k+2\ell+1} (k-\ell-2)! \dif \varphi 
\nonumber
\\
&\hspace{-0.25em}= \frac{C_4^{k+1} \eps^{-2k+1}}{4C_1} \sum_{\ell=1}^{k-1} \binom{k}{\ell} (\ell-1)!(k-\ell-2)! \int_
{0}^\theta e^{C_4(k-1)\varphi} \dif \varphi
\nonumber
\\
&\hspace{-0.25em}\leq
\frac{C_4^k \eps^{-2k+1}}{4C_1} \sum_{\ell=1}^{k-1} \binom{k}{\ell} \frac{(\ell-1)!(k-\ell-2)!}{k-1} e^{C_4(k-1)\theta} \leq \frac{1}{4} C_4^{k} e^{C_4(k-1)\theta}  \eps^{-2k+1}(k-2)! 
\label{eq:induction3}
\end{align}
for every $\theta\in [0,1]$.
Putting together the estimates \eqref{eq:induction1}, \eqref{eq:induction2}, and \eqref{eq:induction3} yields that
\[
f_{k,n,\eps}(\theta) \leq \frac{1}{4}C_4^{k} \eps^{-2k+1}(k-2)! + \frac{1}{2}f_{k,n,\eps}(\theta) + \frac{1}{4}C_4^{k} e^{C_4(k-1)\theta} \eps^{-2k+1} (k-2)!
\]
for every $\theta\in [0,1]$, which rearranges to give that
\[
f_{k,n,\eps}(\theta) \leq C_4^{k} e^{C_4(k-1)\theta} \eps^{-2k+1} (k-2)!
\]
for every $\theta \in [0,1]$ as desired. Since $n\geq 1$ and $0 < \eps \leq \delta$ were arbitrary, taking $n \to\infty $ completes the induction step and hence also the proof.
\end{proof}

\begin{remark}
The correct form of the induction hypothesis \eqref{eq:generating_induction_hypothesis} needed to make this argument work was not at all obvious to us, and was found by extensive trial and error. We would be interested to know if someone is aware of a more systematic way of approaching similar problems.
\end{remark}


\begin{remark} Consider the generating function
\[
\sN_n(s,t,u)= \sN_n(s,t,u;G,p) = \sum_{k=1}^\infty \frac{u^k}{k!}\sF_{k,n}(s,t;G,p),
\]
which satisfies $\partial_u \sN_n=\sM_n$. 
Summing the differential inequality given by \cref{lem:differential_recursion} over $k\geq1 $ implies the partial differential inequality
\begin{equation}
\label{eq:pde}
\partial_t \sN_n \leq \frac{Mpe^t}{1-p}  \sN_n \partial_u \sN_n \qquad \qquad \forall s,t \in \R,\, u \geq 0,\, n \geq 1.
\end{equation}
(Note that while $\sN_n$ need not be differentiable, it is locally Lipschitz and hence differentiable almost everywhere.) See e.g.\ the discussion of exponential generating functions in \cite{MR2172781}. This point of view may be a useful starting point for further analysis. (It appeared to us to be ill-suited to our present aims, however.)
\end{remark}

\section{Completing the proof}
\label{sec:completing}

In this section we complete the proof of \cref{thm:main_radius,thm:main_volume}. It remains to establish lower bounds in the slightly supercritical regime, as well as both upper and lower bounds in the critical and slightly subcritical regimes. Several of these bounds are closely related to estimates that have already been proven in the literature, but still require a delicate treatment to establish in the desired sharp form. 

\medskip

We begin by proving upper bounds in the critical and slightly subcritical regimes under the assumption that $\nabla_{p_c}<\infty$. 

\begin{prop}[Subcritical upper bounds]
\label{prop:subcritical_upper}
Let $G$ be an infinite, connected, locally finite, quasi-transitive graph, and suppose that $\nabla_{p_c}<\infty$. Then there exists positive constants $c$ and $C$ such that
\begin{equation}
\label{eq:radius_subcritical_upper}
\bP_p\left(\operatorname{Rad}_\mathrm{ext}(K_v) \geq r \right) \leq
\bP_p\left(\operatorname{Rad}_\mathrm{int}(K_v) \geq r \right)
\leq \frac{C}{r} \exp\left(-c|p-p_c|r\right)
\end{equation}
and
\begin{equation}
\bP_p\left(|K_v| \geq n \right) 
\leq \frac{C}{n^{1/2}} \exp\left(-c|p-p_c|^2 n\right)
\label{eq:volume_subcritical_upper}
\end{equation}
for every $n,r \geq 1$, $0\leq p \leq p_c$, and $v\in V$.
\end{prop}

Recall that we write $\asymp$, $\preceq$, and $\succeq$ for equalities and inequalities that hold to within multiplication by a positive constant depending only on $G$.

\begin{proof}
Fix $v\in V$ and write $R_v=\operatorname{Rad}_\mathrm{int}(K_v)$. 
As discussed in the introduction, it is known that if $G$ is an infinite, connected, locally finite, quasi-transitive graph with $\nabla_{p_c}<\infty$, then 
\begin{equation}
\label{eq:BAKN}
\bP_{p_c}(|K_v|\geq n) \asymp n^{-1/2} \qquad \text{ and } \qquad \bP_{p_c}(R_v\geq r) \asymp r^{-1}
\end{equation}
for every $n,r \geq 1$, and
\begin{equation}
\label{eq:susceptibility}
\bE_{p}\left[|K_v|\right] \asymp (p-p_c)^{-1}
\end{equation}
for every $0 \leq p \leq p_c$. These results essentially follow from the works of Barsky and Aizenman \cite{MR1127713},  Kozma and Nachmias \cite{MR2551766}, and Aizenman and Newman \cite{MR762034}. These papers all dealt with the case $G=\Z^d$, see \cite[Section 7]{Hutchcroftnonunimodularperc} for a discussion of how to generalize these results to arbitrary quasi-transitive graphs with $\nabla_{p_c}<\infty$. It follows from \eqref{eq:susceptibility} and the tree-graph method of Aizenman and Newman \cite{MR762034} (see also \cite[Chapter 6.3]{grimmett2010percolation}) that there exists a constant $C_1$ such that
\[
\bE_{p}\left[|K_v|^k\right] \leq k! C_1^k |p-p_c|^{-2k+1}
\]
for  every $k\geq 1$ and $p<p_c$ and hence that there exists a constant $c_1=1/2C_1$ such that
\[
\bE_{p}\left[|K_v| e^{c_1|p-p_c|^2|K_v|}\right] \leq \sum_{k=0}^\infty \frac{|p-p_c|^{2k}}{2^k C_1^k k!} (k+1)! C_1^{k+1} |p-p_c|^{-2k-1} =  \sum_{k=0}^\infty C_1 k 2^{-k} |p-p_c|^{-1} \preceq |p-p_c|^{-1}
\]
for every $0\leq p<p_c$.  Markov's inequality then implies that
\begin{equation}
\bP_p\left( |K_v| \geq n\right) \preceq \frac{1}{n(p-p_c)} \exp\left[-c_1|p-p_c|^2n\right]
\label{eq:volume_subcritical_outsidewindow}
\end{equation}
for every $0\leq p<p_c$, and  together with \eqref{eq:BAKN} this implies the desired bound \eqref{eq:volume_subcritical_upper}. (Indeed, simply use the bound \eqref{eq:volume_subcritical_outsidewindow} if $n\geq  (p-p_c)^{-2}$ and the bound \eqref{eq:BAKN} otherwise, noting that $\bP_p(|K_v|\geq n)$ is increasing in $p$.) See also \cite{1901.10363} for an alternative derivation of the inequality \eqref{eq:volume_subcritical_outsidewindow} from \eqref{eq:BAKN}.

\medskip

It remains to prove \eqref{eq:radius_subcritical_upper}. The case $r \leq |p-p_c|^{-1}$ is already handled by \cref{lem:inside_window}, so it suffices to consider the case $r \geq |p-p_c|^{-1}$. We have by the union bound that
\[
\bP_p\left(R_v \geq r\right)
\leq 
\bP_p\left(|K_v| \geq |p-p_c|^{-1}r\right) +
\bP_p\left(R_v \geq r \text{ and }
|K_v| \leq |p-p_c|^{-1}r
 \right).
\]
Using \eqref{eq:volume_subcritical_outsidewindow} to bound the first term and \cref{lem:SkinnyRadius} with $\lambda=|p-p_c|$ to bound the second yields that there exist positive constant $c_2$ such that
\[
\bP_p\left(R_v \geq r\right)
\preceq
\frac{1}{r} \exp\left[-c_1|p-p_c|r\right]
+ \left(\frac{1}{r}+|p-p_c|\right) \exp\left[-c_2 |p-p_c|r\right],
\]
which is easily seen to be of the required order (since $x e^{-xr} \leq 2e r^{-1} e^{-xr/2}$ for every $x \in \R$).
\end{proof}

\medskip

We next study the intrinsic radius in the subcritical case. This is our only bound that holds for \emph{all} quasi-transitive graphs.

\begin{prop}
\label{lem:subcritical_radius_lower}
Let $G$ be an  infinite, connected, locally finite, quasi-transitive graph. Then there exist positive  constants $c$, $C$, and $\delta$  such that
\[
\bP_p\left(\operatorname{Rad}_\mathrm{int}(K_v) \geq r \right)
\geq \frac{c}{r} \exp\left(-C|p-p_c|r\right)
\]
for every $p\in (p_c-\delta,p_c]$, $r\geq 1$, and $v\in V$.
\end{prop}

\begin{proof}
A similar argument to that of \cref{lem:inside_window} establishes that 
\[
\bP_p\left(\operatorname{Rad}_\mathrm{int}(K_v) \geq r \right) \geq \bP_q\left(\operatorname{Rad}_\mathrm{int}(K_v) \geq r \right) \exp\left[-\frac{q-p}{p} r\right]
\]
for every $0\leq p \leq q \leq 1$ and $r \geq 1$. It follows in particular that
\[
\bP_p\left(\operatorname{Rad}_\mathrm{int}(K_v) \geq r \right) \geq  \exp\left[-\frac{(1-p)}{p} r\right]
\]
for every $0\leq p \leq 1$ and $r\geq 1$, which implies the claim in the case $p_c=1$.
On the other hand, if $p_c<1$ then we have that there exists a constant $c$ such that
$\bP_p(v \to \infty) \geq c(p-p_c)$
for every $p_c \leq p \leq 1$ \cite{aizenman1987sharpness,duminil2015new,1901.10363}, so that
\[
\bP_p\left(\operatorname{Rad}_\mathrm{int}(K_v) \geq r \right) \geq c(q-p_c)\exp\left[-\frac{q-p}{p} r\right]
\]
for every $0\leq p \leq p_c \leq q \leq 1$. Taking $q -p_c = (p_c-p) \wedge r^{-1}$ implies the claim.
\end{proof}

We next prove sharp lower bounds on the tail of the volume under the assumption that $\nabla_{p_c}<\infty$.

\begin{prop}
\label{lem:volume_lower}
Let $G$ be an  infinite, connected, locally finite, quasi-transitive graph, and suppose that $\nabla_{p_c}<\infty$. Then there exist positive  constants $c$, $C$, and $\delta$  such that
\begin{equation}
\label{eq:volume_lower}
\bP_p\left(n \leq |K_v| < \infty \right)
\geq c n^{-1/2} \exp\left(-C|p-p_c|^2 n\right)
\end{equation}
for every $p \in (p_c-\delta,p_c+\delta)$, $r\geq 1$, and $v\in V$.
\end{prop}

\begin{remark}
\label{rem:nguyen}
Nguyen \cite{MR923855} proved that, under the same conditions as \cref{lem:volume_lower}, there exist constants $(C_k)_{k\geq 1}$ such that
\[
\bE_p\left[|K_v|^k\right] \geq C_k |p-p_c|^{-2k+1}
\]
for every $0 \leq p <p_c$ and $k\geq 1$. This is sufficient to determine the value of the gap exponent $\Delta=2$. However, it seems that the argument of \cite{MR923855} does not give sharp ($C_k \geq k!e^{- O(k)}$) control of the value of the constant $C_k$, and therefore does  not establish the subcritical case of the bound \eqref{eq:volume_lower}. Similarly, classical arguments of Durrett and Nguyen \cite{durrett1985thermodynamic} and Newman \cite{MR869320} can be used to prove related inequalities for the truncated $k$th moment $\bE_{p,\infty}\left[|K_v|^k\right]$ in the slightly supercritical regime. Again, however, it appears that these estimates are not sharp, and lose various logarithmic factors compared to our  estimate \eqref{eq:volume_lower}. 
\end{remark}

\begin{proof}[Proof of \cref{lem:volume_lower}]
Write $R_v=\Rad_\mathrm{int}(K_v)$. 
First suppose that $p \leq p_c$. Taking $\lambda = \alpha |p-p_c|$ in \cref{lem:SkinnyRadius}, we obtain that that there exist positive constants $c_1$ and $C_1$ such that
\[
\bP_p\left(|K_v| \leq n, R_v \geq \alpha |p-p_c| n\right) \leq C_1 \left(\frac{1}{\alpha |p-p_c| n} + \alpha |p-p_c| \right) \exp\left[-c_1  \alpha^2 |p-p_c|^2 n \right]
\]
for every $0 \leq p \leq p_c$, $n\geq 1$, and $\alpha \geq 1$. Letting  $c_2$, $C_2$, and $\delta_1$ be the constants from \cref{lem:subcritical_radius_lower}, it follows that
\begin{align*}
\bP_p\left( |K_v| \geq n\right) 
&\geq 
\bP_p\left(R_v \geq \alpha |p-p_c| n  \right) 
- \bP_p\left(|K_v| \leq n, R_v \geq \alpha |p-p_c| n\right) 
\\
&\geq
\frac{c_2}{\alpha |p-p_c| n}\exp\left[-C_2\alpha |p-p_c|^2 n\right]
\\&\hspace{4.9cm}- C_1 \left(\frac{1}{\alpha |p-p_c| n} + \alpha |p-p_c| \right) \exp\left[-c_1  \alpha^2 |p-p_c|^2 n \right]
\end{align*}
for every $p_c-\delta_1 \leq p \leq p_c$, $r\geq 1$, and $\alpha \geq 1$. Taking $\alpha = 1\vee(2C_1/c_1)$ we deduce that there exist positive constants $c_3$, $C_3$, and $C_4$ such that
\begin{align*}
\bP_p\left(|K_v| \geq n \right) 
\geq 
\frac{c_3}{ |p-p_c| n}\exp\left[-C_3 |p-p_c|^2 n\right]
- C_4 \left(\frac{1}{|p-p_c| n} + |p-p_c| \right) \exp\left[-2C_3 |p-p_c|^2 n \right]
\end{align*}
for every $p \in (p_c-\delta,p_c)$ and $n \geq 1$. It follows readily that there exist positive constants $c_4$ and $C_5$ such that
\begin{align*}
\bP_p\left(|K_v| \geq n \right) 
\geq 
\frac{c_4}{ \sqrt{n}}\exp\left[-C_3 |p-p_c|^2 n\right]
\end{align*}
for every $p \in (p_c-\delta,p_c)$ and $n \geq C_5|p-p_c|^{-2}$. Since $\bP_p(|K_v|\geq n)$ is decreasing in $n$, it follows that
\begin{align}
\label{eq:subcritical_volume_lower_outside_window}
\bP_p\left(|K_v| \geq n \right) 
\geq 
\frac{c_4}{ \sqrt{n \vee C_5|p-p_c|^{-1}}}\exp\left[-C_3 |p-p_c|^2 (n \vee C_5|p-p_c|^{-1})\right]
\end{align}
for every $p \in (p_c-\delta,p_c)$ and $n\geq 1$.

\medskip

We now handle the case that $p \leq p_c$ and $n$ is of order at most $|p-p_c|^{-2}$. It follows from the proof
  of \cite[Proposition 3.6]{hutchcroft20192} that
\[
\sup_{u\in V}\bE_{p_c}\left[\sum_{\ell =r}^{2r}\# \partial B_\mathrm{int}(u,\ell) \right] \geq r+1 
\]
for every $r\geq 1$, and an argument similar to that performed in the proof of \cref{lem:inside_window} shows that there exist constants $\delta_4 \leq p_c/2$ and $C_6$ such that
\begin{equation}
\label{eq:expected_volume_window}
\sup_{u\in V}\bE_{p}\left[\sum_{\ell =r}^{2r}\# \partial B_\mathrm{int}(u,\ell) \right] \geq (r+1)\exp\left[-\frac{2(p-p_c)}{p} r\right] \geq (r+1)\exp\left[-C_6 |p-p_c| r\right]
\end{equation}
for every $p\in (p_c-\delta_4,p_c]$ and $r\geq 1$. Applying \eqref{eq:BAKN}, it follows that
\begin{equation}
\sup_{u\in V}\bE_{p}\left[\sum_{\ell =r}^{2r}\# \partial B_\mathrm{int}(u,\ell) \mid R_u \geq r \right] \geq c_5 (r+1)^2 \exp\left[-C_6 |p-p_c| r\right]
\label{eq:expected_volume_window2}
\end{equation}
for every $p\in (p_c-\delta_4,p_c]$ and $r\geq 1$. On the other hand, 
since $\nabla_{p_c}<\infty$, it is known \cite{MR2551766,sapozhnikov2010upper} that there exists a constant $C_7$ such that
\[
\bE_{p}\left[\#B_\mathrm{int}(u,r)\right] \leq C_7 (r+1)
\]
for every $u\in V$, $0 \leq p \leq p_c$ and $r\geq 0$. A straightforward and well-known variation on the tree-graph inequality method of Aizenman and Newman \cite{MR762034} (see e.g.\ \cite[Lemma 2]{kozma2011percolation}) gives that
\[
\bE_{p}\left[(\#B_\mathrm{int}(u,2r))^2\right] \leq \sup_{w\in V} \bE_{p}\left[(\#B_\mathrm{int}(w,2r))\right]^3
\]
for every $u\in V$ and $r\geq 1$, and hence that 
\[
\bE_{p}\left[(\#B_\mathrm{int}(u,2r))^2\right] \leq  C_7^3(r+1)^3
\]
for every $r\geq 1$, $u \in V$ and $0\leq p \leq p_c$. 
It follows from the Paley-Zygmund inequality that
\begin{equation}
\label{eq:PaleyZygmund}
\bP_p\left(|K_u| \geq \frac{1}{2}\bE_{p}\left[\sum_{\ell =r}^{2r}\# \partial B_\mathrm{int}(u,\ell) \mid R_u \geq r\right] \right) \geq \myfrac[0.5em]{\bE_{p}\left[\sum_{\ell =r}^{2r}\# \partial B_\mathrm{int}(u,\ell) \right]^2}{4\bE_{p}\left[(\#B_\mathrm{int}(u,2r))^2\right]}
\end{equation}
for every $u\in V$, $0<p \leq p_c$, and $r \geq 1$. Applying this inequality together with \eqref{eq:expected_volume_window}, \eqref{eq:expected_volume_window2}, and \eqref{eq:PaleyZygmund} and maximizing over $u$, it follows that
\[
\sup_{u\in V}\bP_p\left(|K_v| \geq \frac{c_5}{2}(r+1)^2e^{-C_6|p-p_c|r}\right) \geq \frac{e^{-2C_6|p-p_c|r}}{4 C_7^3 (r+1)}
\]
for every $p \in (p_c-\delta_4,p_c]$ and $r\geq 1$. Since $G$ is connected and quasi-transitive, it follows 
straightforwardly that there exist constants $c_6$, $c_7$, and $c_8$ such that
\[
\bP_p\left(|K_v| \geq n \right) \geq c_6
\sup_{u\in V} \bP_p\left(|K_u| \geq n \right) \geq \frac{c_7}{\sqrt{n}}
\]
for every $p\in (p_c-\delta_4,p_c]$ and $1 \leq n \leq c_8|p-p_c|^{-2}$. The claimed bound \eqref{eq:volume_lower} follows in the case $p \in (p_c-\delta_4,p_c]$ from this together with \eqref{eq:subcritical_volume_lower_outside_window}.

\medskip

We now consider the case $p \geq p_c$. Let $\omega_p$ and $\omega_{p_c}$ be Bernoulli-$p$ and Bernoulli-$p_c$ percolation on $G$ coupled in the standard monotone way, so that, conditional on $\omega_{p_c}$, every $\omega_{p_c}$-open edge is $\omega_p$ open and every $\omega_{p_c}$-closed edge is chosen to be either $\omega_p$-open or $\omega_p$-closed independently at random with probability $(p-p_c)/(1-p_c)=O(p-p_c)$ to be $\omega_p$-open. Let $K_v^{p_c}$ and $K_v^p$ denote the clusters of $v$ in $\omega_{p_c}$ and $\omega_p$ respectively. 
By \eqref{eq:BAKN}, there exist constants $c_9$ and $C_8$ such that
\[
\P(n \leq |K_v^{p_c}| \leq \alpha n) \geq \frac{c_9}{\sqrt{n}} - \frac{C_8}{\sqrt{\alpha n}}
\]
for every $n \geq 1$ and $\alpha \geq 1$. Taking $\alpha =C_9:= 1 \vee (2C_8/c_9)^2$, it follows that there exists a positive constant $c_{10}$ such that
\begin{equation}
\P(n \leq |K_v^{p_c}| \leq C_9 n) \geq \frac{c_{10}}{\sqrt{n}} 
\label{eq:critical_volume_one_scale}
\end{equation}
for every $n\geq 1$. Let $\sA_n$ be the event that $n \leq |K_v^{p_c}| \leq C_9n$ and let $\sB_n$ be the event that $n \leq |K_v^p|<\infty$. If $\sA_n$ occurs but $\sB_n$ does not, then there must exist an $\omega_{p_c}$-closed edge in the boundary of $K_v^{p_c}$ that is $\omega_p$-open and whose other endpoint is connected to infinity in $\omega_p$ by an open path that does not visit any vertex of $K_v^{p_c}$. Conditional on $K_v^{p_c}$, the probability that any particular edge in the boundary of $K_v^{p_c}$ has this property is bounded by $(p-p_c)\theta^*(p)/(1-p_c)=O((p-p_c)^2)$, and it follows by the FKG inequality that there exists a constant $C_{10}$ such that
\begin{equation}
\P(\sB_n \mid K_v^{p_c}) \geq \mathbbm{1}(\sA_n) \left[ 1-1\wedge \frac{(p-p_c)\theta^*(p)}{1-p_c}\right]^{M |K_v^{p_c}|} \geq \mathbbm{1}(n \leq |K_v^{p_c}|\leq C_9n)e^{-C_{10} (p-p_c)^2 n},
\label{eq:moving_p_volume}
\end{equation}
where $M$ is the maximum degree of $G$. The claimed bound follows from \eqref{eq:critical_volume_one_scale} and \eqref{eq:moving_p_volume} by taking expectations over $K_v^{p_c}$.
\end{proof}



Finally, we prove a lower bound on the tail of the radius of a finite cluster in the supercritical regime under the assumption that $p_c<p_{2\to 2}$.

\begin{prop}
\label{prop:radius_lower_bounds}
Let $G$ be an infinite, connected, locally finite, quasi-transitive graph, and suppose that $p_c<p_{2\to 2}$. Then there exist positive constants $c$ and $C$ such that
\begin{equation}
\label{eq:radius_lower}
\bP_p\left(\operatorname{Rad}_\mathrm{int}(K_v) \geq r \right)
\geq \bP_p\left(\operatorname{Rad}_\mathrm{ext}(K_v) \geq r \right) \geq \frac{c}{r} \exp\left(-C|p-p_c|r\right)
\end{equation}
for every $r \geq 1$ and $p\in(p_c-\delta,p_c+\delta)$.
\end{prop}

\begin{proof}[Proof of \cref{prop:radius_lower_bounds}]
By \cref{lem:subcritical_radius_lower}
there exist positive constants $c_1$, $C_1$, and $\delta_1$ such that
\begin{equation}
\label{eq:radius_lower}
\bP_p\left(\operatorname{Rad}_\mathrm{int}(K_v) \geq r \right) \geq \frac{c_1}{r} \exp\left(-C_1|p-p_c|r\right)
\end{equation}
for every $r \geq 1$ and $p\in(p_c-\delta_1,p_c)$. On the other hand, it follows from \cite[Proposition 3.2]{hutchcroft20192} that 
\[
\bP_{p}\left(\operatorname{Rad}_\mathrm{int}(K_v) \geq r \text{ and } \operatorname{Rad}_\mathrm{ext}(K_v) \leq \ell \right) \leq 3\|T_p\|_{2\to 2}\exp\left[-\frac{r}{e\|T_p\|_{2\to 2}} \right] |B(v,\ell)|^{1/2}
\]
for every $0 \leq p < p_{2\to 2}$ and $r,\ell \geq 1$. Since $p_c<p_{2\to 2}$, it follows that there exist constants $c_2$, $c_3$, $C_2$ and $\delta_2$ such that
\begin{equation}
\label{eq:ballisticity}
\bP_{p}\left(\operatorname{Rad}_\mathrm{int}(K_v) \geq r \text{ and } \operatorname{Rad}_\mathrm{ext}(K_v) \leq c_2 r \right) \leq C_2 e^{-c_3 r}
\end{equation}
for every $0\leq p \leq p_c+\delta_2$ and $r \geq 1$. It  follows in particular that there exist positive constants $c_4$ and $r_0$ such that
\begin{align}
\bP_{p}\left(\operatorname{Rad}_\mathrm{ext}(K_v) \geq c_2 r \right) 
&\geq  \bP_{p}\left(\operatorname{Rad}_\mathrm{int}(K_v) \geq r \right)- \bP_{p}\left(\operatorname{Rad}_\mathrm{int}(K_v) \geq r \text{ and } \operatorname{Rad}_\mathrm{ext}(K_v) \leq c_2 r \right)
\nonumber
\\
&\geq \frac{c_1}{r} \exp\left(-C_1|p-p_c|r\right) - C_2e^{-c_3 r} 
\geq \frac{c_4}{r} \exp\left(-C_1|p-p_c|r\right)
\label{eq:radius_subcritical_lower5}
\end{align}
for every $p\in (p_c-\delta_2,p_c]$ and $r  \geq r_0$. This is easily seen to imply \eqref{eq:radius_lower} in the case $p\in (p_c-\delta_3,p_c]$.

We now treat the supercritical case.
Combining the inequality \eqref{eq:radius_subcritical_lower5} with \eqref{eq:vol_supercritical_upper}, an easy argument similar to that of the previous paragraph shows that there exist positive constants $c_5$, $C_3$, and $C_4$ such that
\[
\bP_{p}\left(\operatorname{Rad}_\mathrm{ext}(K_v) \geq r \text{ and } |K_v| \leq C_3|p-p_c|^{-1}r \right) \geq \frac{c_5}{r}\exp\left(-C_4|p-p_c|r\right)
\]
for every $p\in (p_c-\delta_2,p_c)$. We apply a similar coupling argument to the end of the proof of \cref{lem:volume_lower}, with the important difference that we compare $(p_c+\eps)$-percolation to $(p_c-\eps)$-percolation rather than to $p_c$-percolation.  Let $0<\eps \leq \delta_2$, let $p=p_c+\eps$, and let $q=p_c-\eps$. Let $\omega_p$ and $\omega_q$ be Bernoulli-$p$ and Bernoulli-$q$ percolation on $G$ coupled in the standard monotone way, so that, conditional on $\omega_q$, every $\omega_q$-open edge is $\omega_p$ open and every $\omega_q$-closed edge is chosen to be either $\omega_p$-open or $\omega_p$-closed independently at random with probability $(p-q)/(1-q)=O(\eps)$ to be $\omega_p$-open. Let $K_v^q$ and $K_v^p$ denote the clusters of $v$ in $\omega_q$ and $\omega_p$ respectively. Let $\sA_r$ be the event that $K_v^q$ has extrinsic radius at least $r$ and volume at most $C_3 \eps^{-1}r$, and let $\sB_r$ be the event that $K_v^p$ is finite and has extrinsic radius at least $r$. If $\sA_r$ occurs but $\sB_r$ does not, then there must exist an $\omega_q$-closed edge in the boundary of $K_v^q$ that is $\omega_p$-open and whose other endpoint is connected to infinity in $\omega_p$ by an open path that does not visit any vertex of $K_v^q$. Conditional on $K_v^q$, the probability that any particular edge in the boundary of $K_v^q$ has this property is bounded by $(p-q)\theta^*(p)/(1-q)=O(\eps^2)$, and it follows by the FKG inequality that there exists a constant $C_5$ such that
\[
\P(\sB_r \mid K_v^q) \geq \mathbbm{1}(\sA_r) \left[1- 1\wedge \frac{(p-q)\theta^*(p)}{1-q}\right]^{M |K_v^q|} \geq \mathbbm{1}(\sA_r)e^{-C_5 \eps r},
\]
where $M$ is the maximum degree of $G$ and where we used that $|K_v^q| \leq C_3 \eps^{-1} r$ on the event $\sA_r$ in the second inequality. Taking expectations, it follows that
\[
\P(\sB_r) \geq \P(\sA_r)e^{-C_5 \eps r} \geq \frac{c_5}{r}e^{-(C_4+C_5)\eps r}
\]
and hence that there exists a constant $C_6$ such that
\begin{equation}
\bP_p\left(r \leq \operatorname{Rad}_\mathrm{int}(K_v) <\infty \right)
\geq \bP_p\left(r \leq \operatorname{Rad}_\mathrm{ext}(K_v) <\infty \right) \geq \frac{c_5}{r} \exp\left(-C_6|p-p_c|r\right)
\end{equation}
for every $p\in (p_c,p_c+\delta_2)$ and $r\geq 1$. This completes the proof. (Note that this argument cannot be applied directly to the intrinsic radius as written due to non-monotonicity issues.) \qedhere

\end{proof}

We now have all the ingredients required to conclude the proofs of our main theorems.

\begin{proof}[Proof of \cref{thm:main_volume}]
The upper bound follows from \cref{prop:vol_supercritical_upper,prop:subcritical_upper}, while the lower bound follows from \cref{lem:volume_lower}.
\end{proof}

\begin{proof}[Proof of \cref{thm:main_radius}]
The upper bound follows from \cref{prop:int_rad_supercritical_upper,prop:subcritical_upper}, while the lower bound follows from \cref{prop:radius_lower_bounds}.
\end{proof}

\section{Perspectives on the Euclidean case}
\label{subsec:Euclidean}

In this subsection we discuss the (apparently rather substantial) challenges that remain to extend our analysis from nonamenable graphs to the high-dimensional Euclidean setting, and give some perspectives on how these challenges might be overcome.

Let us begin by stating what is conjectured to be the case. Let $d \geq 7$ and consider the hypercubic lattice $\Z^d$. The conjectured analogue of \cref{thm:main_volume} is that there exists $\delta>0$ such that
\begin{align}
\bP_{p}(n \leq |K| < \infty)  &\asymp \begin{cases} 
n^{-1/2} \exp\left[ -\Theta\left(|p-p_c|^2 n\right)\right] & p \in (p_c-\delta,p_c)\\
n^{-1/2} & p=p_c\\
n^{-1/2} \exp\left[ -\Theta\left(\left(|p-p_c|^2 n\right)^{(d-1)/d}\right)\right] & p\in (p_c,p_c+\delta),
\end{cases}
\label{eq:Zd_volume}
\end{align}
while the conjectured analogue of \cref{thm:main_radius} is that
\begin{align}
\bP_{p}(r \leq \Rad_\mathrm{int}(K) < \infty)  &\asymp 
n^{-1} \exp\left[ -\Theta\left(|p-p_c| n\right)\right] 
\label{eq:Zd_intrinsic}
\intertext{and}
\bP_{p}(r \leq \Rad_\mathrm{ext}(K) < \infty)  &\asymp n^{-2} \exp\left[ -\Theta\left(|p-p_c|^{1/2} n\right)\right] 
\label{eq:Zd_extrinsic}
\end{align}
for all $p\in (p_c-\delta,p_c+\delta)$. 
Further related questions of interest include the behaviour of the truncated two-point function $\hat \tau_p (x,y) = \bP(x \leftrightarrow y, x \nleftrightarrow \infty)$, which is conjectured to satisfy
\begin{equation}
\label{eq:Zd_two_point}
\hat \tau_p(x,y) \asymp \|x-y\|^{-d+2} \exp\left[ - \Theta\left(|p-p_c|^{1/2}\|x-y\|\right)\right] 
\end{equation}
for all $p \in (p_c-\delta,p_c+\delta)$ and $x,y \in \Z^d$. In particular, it is conjectured that the \textbf{correlation length} $\xi(p)$ satisfies
\begin{equation}
\xi(p)^{-1} := -\lim_{n\to \infty} \frac{1}{n} \log \sup \left\{ \bP_p(0 \leftrightarrow x, 0 \nleftrightarrow \infty) : x \in \Z^d, \|x\|\geq n\right\} \asymp |p-p_c|^{-1/2}
\label{eq:Zd_correlation_length}
\end{equation}
for $p \in (p_c-\delta,p_c+\delta)$. (Note that \eqref{eq:Zd_two_point} would trivially imply \eqref{eq:Zd_correlation_length}.) Besides their intrinsic interest, a solution to these conjectures may be a necessary prerequisite to understanding  invasion percolation, the minimal spanning forest, and random walks on slightly supercritical clusters. See \cite[Part IV]{heydenreich2015progress} for an overview.

At present, the state of these conjectures can be summarised as follows: The $p=p_c$ cases of \eqref{eq:Zd_volume} and \eqref{eq:Zd_intrinsic} were proven to hold for all quasi-transitive graphs satisfying the triangle condition by Barsky and Aizenman \cite{MR1127713} and Kozma and Nachmias \cite{MR2551766}, respectively. We showed how these statements imply the subcritical cases of the same statements in \cref{prop:subcritical_upper,lem:subcritical_radius_lower,lem:volume_lower}. Hara and Slade \cite{MR1043524} proved via the lace expansion that the triangle condition holds on $\Z^d$ for sufficiently large $d$, as well as for ``spread out'' models in dimension $d \geq 7$. Around the same time, Hara built upon the methods of \cite{MR1043524} to prove the $p \leq p_c$ case of \eqref{eq:Zd_correlation_length} under the same hypotheses, i.e., that $\xi(p) \asymp (p-p_c)^{-1/2}$ as $p \uparrow p_c$. Later, Hara, van der Hofstad, and Slade \cite{MR1959796} performed a `physical space' version of the lace-expansion that allowed them to prove the $p=p_c$ case of \eqref{eq:Zd_two_point} under the same hypotheses. Kozma and Nachmias \cite{MR2748397} then applied this result to prove the $p=p_c$ case of \eqref{eq:Zd_extrinsic}.  The subcritical case of \eqref{eq:Zd_extrinsic} has very recently been established in the independent works \cite{hutchcroft2021high,chatterjee2021subcritical}, with \cite{hutchcroft2021high} also establishing the upper bound of \eqref{eq:Zd_two_point} in the subcritical case.
In contrast, almost no progress has been made on the slightly supercritical cases of these conjectures.

As we stated in the introduction, we are optimistic that some of the techniques we have developed in this paper will be prove useful to the eventual solution of these conjectures. We now outline some ideas about what such a solution might look like. 
Note that several of the challenges one would need to overcome to adapt our methods to the high-dimensional Euclidean setting are of a similar nature to those one would need to overcome to solve the more qualitative problems stated in \cite[Section 5.3]{HermonHutchcroftSupercritical}.
\begin{enumerate}
\item A good first step would be to find a sharp bound on the negative part of the derivative $\bD_{p,n}|K|^k$ for $p$ slightly supercritical. Such a bound would need to be of order $C^k (k!)^{d/(d-1)}|p-p_c|^{-2k}$, but it is unclear what form it should take, presumably being written in terms of some higher truncated moment. A potentially serious difficulty is that it seems one cannot rely on a worst case analysis of the expected number of edges connecting some deterministic set $S$ to infinity off of $S$, as we did in the proof of \cref{prop:NegativeTermQuant}. Indeed, heuristically, if $\Lambda_n = [-n,n]^d$ is a box with $n = \Omega(\xi(p)) = \Omega((p-p_c)^{-1/2})$ then the typical number of edges in the boundary of $\Lambda_n$ whose other endpoint is connected to $\infty$ off of $\Lambda_n$ should be of order $(p-p_c)^{3/2} |\Lambda_n|^{(d-1)/d}$, where $(p-p_c)^{3/2}$ is conjectured to be the order of the probability that the origin is connected to infinity inside a half-space. See \cite{chatterjee2018restricted} for various related rigorous results. This (presumably) worst case bound would be too small to lead to a proof of \eqref{eq:Zd_volume}, even if one did not have the positive term to contend with. Thus, to bound $\bD_{p,n}$ via this approach, one would need to somehow understand how the geometry of large finite clusters in slightly supercritical percolation leads them to have a greater number of pivotal connections to infinity in their boundary than a box of comparable volume would. The techniques developed to understand phenomena such as Wulff crystals in supercritical percolation may be relevant \cite{MR2241754}.

An alternative approach may be to use the \emph{OSSS inequality}, due to  O’Donnel, Saks,
Schramm, and Servedio \cite{o2005every}, which has recently been recognised as a powerful tool in the study of percolation and other models following the breakthrough work of Duminil-Copin, Raoufi, and Tassion \cite{MR3898174,1705.07978}; see also \cite{1901.10363,hutchcroft2021high} for applications to the critical behaviour of  Bernoulli percolation. Briefly, this inequality lets us prove differential inequalities by finding randomized algorithms that determine the value of the function whose expectation we are interested in but which have a low maximum \emph{revealment}, that is, a low maximum probability of querying whether any particular edge is open or closed.  While this inequality is most powerful as a tool for studying monotone functions, it can also be used to bound the expected \emph{total} number of pivotals for non-monotone functions, which would mean bounding the sum $\bD_{p,n}+\bU_{p,n}$ in our context. Such a bound would in fact be just as viable in the remainder of our strategy as a bound on $\bD_{p,n}$ itself. The difficulty with this approach is to find, say, a low-revealment algorithm determining whether or not the origin is in a large finite cluster. It is unclear how this might be done. One possibility is to use invasion percolation, but this may be putting the cart before the horse; it seems that invasion percolation should be even harder to analyse than slightly supercritical percolation itself. 
\item Even if one is able to get good bounds on $\bD_{p,n}$ or $\bD_{p,n}+\bU_{p,n}$, there remains the substantial challenge of getting good upper bounds on $\bU_{p,n}$ in the manner of \eqref{eq:overview_hope}. It is possible that this could be done by methods that are rather similar to what we have done in \cref{subsec:positive_term,subsec:positivetermI,subsec:positivetermII}. However, it is likely that, due to the different form of the lower bound on the negative term, one would need to initiate this analysis by proving a version of our skinny clusters estimate in which one could profitably take the radius to be at least a \emph{power} of the radius rather than a small multiple as we have done here. Bounds of this form are known for Galton-Watson trees \cite{MR3077536,MR3916103}, but it seems unclear what one could hope to be true for high dimensional lattices, or how such an estimate might be proven. If such a bound on skinny clusters were found, we are hopeful that an analysis very similar to that performed in \cref{subsec:positivetermI,subsec:positivetermII} could be used derive the higher-order variants of this bound needed to bound $\bU_{p,n}|K|^k$.
\end{enumerate}
Finally, we remark that, by analogy with our setting, it may be substantially easier to obtain the correct behaviour for the intrinsic radius than for the volume. 

\section*{Glossary of notation}
\label{sec:Glossary}

\begin{enumerate}[leftmargin=3cm]
\item[$\asymp$, $\preceq$, $\succeq$] Equalities and inequalities holding to within positive multiplicative constants depending only on the choice of graph.
\item[$\mathbf{P}_p=\mathbf{P}_p^G$] The law of Bernoulli-$p$ bond percolation on $G$. Defined in \cref{sec:intro}.
\item[$p_c$, $p_{2\to 2}$] The critical probability and $L^2$ boundedness threshold. The critical probability $p_c$ always refers to $p_c(G)$, rather than $p_c$ of any subgraph $H$ of $G$, unless specified otherwise. Defined in \cref{sec:intro}.
\item[$T_p$] The two-point matrix $T_p(u,v)=\bP_p(u\leftrightarrow v)$. Defined in \cref{sec:intro}.
\item[$K_v$] The cluster of $v$. When $v$ is fixed we often write $K=K_v$. Defined in \cref{subsec:intro_finite}.
\item[$\Rad_\mathrm{int}(K_v)$] The intrinsic radius of $K_v$, that is, the maximum intrinsic distance from $v$ to another point of $K_v$. Defined in \cref{subsec:intro_finite}.
\item[$\Rad_\mathrm{ext}(K_v)$] The extrinsic radius of $K_v$, that is, the maximum extrinsic distance from $v$ to another point of $K_v$. Defined in \cref{subsec:intro_finite}.
\item[$\nabla_p(v)$] The triangle sum at $v$. Defined in \cref{subsec:intro_finite}.
\item[$R_v$] Shorthand for $\Rad_\mathrm{int}(K_v)$. Defined in \cref{sec:inside_window}.
\item[$E_v$] The number of edges touching $K_v$. Defined in \cref{sec:inside_window}.
\item[$B_\mathrm{int}(v,n)$] The set of vertices of $K_v$ of intrinsic distance at most $n$ from $v$. Defined in \cref{sec:inside_window}.
\item[$\partial B_\mathrm{int}(v,n)$] The set of vertices of $K_v$ of intrinsic distance exactly $n$ from $v$. Defined in \cref{sec:inside_window}.
\item[$\bE_{p,n}$] The expectation truncated when the cluster of $v$ touches more than $n$ edges. Defined in \cref{subsec:setup}.
\item[$\bU_{p,n}$] The positive part of the derivative of $\bE_{p,n}$. Defined in \cref{subsec:setup}.
\item[$-\bD_{p,n}$] The negative part of the derivative of $\bE_{p,n}$. Defined in \cref{subsec:setup}.
\item[$\Bridges(v_1,\ldots,v_k;K_v)$] The number of edges in the subtree of the tree of two-connected components of $K_v$ spanned by the union of the geodesics between the vertices in this tree corresponding to $v_1,\ldots,v_k$. Defined in \cref{subsec:positivetermI}.
\end{enumerate}
\noindent Definitions of generating functions (all defined in \cref{subsec:positivetermI}):
\begin{align*}
 \mathscr{G}_{k,n}(s,t;G,v,p) 
&=
\sum_{a=0}^\infty \sum_{b=0}^\infty \sum_{\substack{x_1,\ldots,x_k\\ \in V(G)}} \bP_{p,n}^G\Bigl(x_1,\ldots,x_k \in K_v, E_v = a, \Bridges(v,x_1,\ldots,x_k;K_v) =b\Bigr) e^{sa+tb}.
\\
\mathscr{F}_{k,n}(s,t;G,p) &= \sup\left\{\mathscr{G}_{k,n}(s,t;H,u,p) : H \text{ a subgraph of $G$, $u$ a vertex of $H$} \right\}.
\\
\sM_n(s,t,u;G,p) &= \sum_{k=0}^\infty \frac{u^k}{k!} \sF_{k+1,n}(s,t;G,p).
\end{align*}

\subsection*{Acknowledgments} We thank Jonathan Hermon and Asaf Nachmias for many helpful discussions, and thank Remco van der Hofstad for helpful comments on an earlier version of this manuscript. We also thank Antoine  Godin for sharing his simplified proof of \cref{prop:NegativeTermQuant} with us and thank the anonymous referees for their careful reading of the manuscript.

\addcontentsline{toc}{section}{References}

 \setstretch{1}
 \footnotesize{
  \bibliographystyle{abbrv}
  \bibliography{unimodularthesis.bib}
  }
\end{document}